\documentclass[11pt]{amsart} 

\usepackage{stmaryrd}
\usepackage{tikz-cd}
\usepackage[color=orange!10]{todonotes}
\usepackage{paralist}

\usepackage[USenglish]{babel}
\usepackage{csquotes}
\usepackage{enumerate}

\usepackage[width=5.65in,height=8.0in,centering]{geometry}
\usepackage[colorlinks,urlcolor=blue]{hyperref}

\usepackage{tikz}
\usetikzlibrary{arrows,cd}

\usepackage{dsfont,mathrsfs}
\usepackage{mathrsfs}

\usepackage[backend=biber,style=alphabetic,maxbibnames=4,doi=false,isbn=false,url=false]{biblatex}
\bibliography{blombda.bib}

\usepackage{mathtools}
\usepackage{amsmath,amssymb,amsthm}
\usepackage{thmtools,thm-restate}

\numberwithin{equation}{section}

\newcommand{\myendsymbol}{\ensuremath{\diamondsuit}}

\numberwithin{equation}{section}

\declaretheorem[
  style=definition,
  title=Example,
  qed={$\myendsymbol$},
  numberlike=equation
]{exa}

\declaretheorem[
  style=definition,
  title=Definition,
  qed={$\myendsymbol$},
  sibling=exa
]{dfn}

\declaretheorem[
  style=definition,
  title=Notation,
  qed={$\myendsymbol$},
  sibling=exa
]{ntn}

\declaretheorem[
  style=definition,
  title=Remark,
  qed={$\myendsymbol$},
  sibling=exa
]{rmk}

\declaretheorem[
  style=plain,
  title=Lemma,
  qed={},
  sibling=exa
]{lem}

\declaretheorem[
  style=plain,
  title=Theorem,
  qed={},
  sibling=exa
]{thmx}

\declaretheorem[
  style=plain,
  title=Corollary,
  qed={},
  sibling=exa
]{cor}

\declaretheorem[
  style=plain,
  title=Proposition,
  qed={},
  sibling=exa
]{prp}

\newcommand{\noproof}{\popQED{\ensuremath{\Box}}}

\DeclareMathOperator{\codim}{codim}   
\DeclareMathOperator{\diag}{{diag}}   
\DeclareMathOperator{\Hom}{Hom}       
\DeclareMathOperator{\id}{id}         
\DeclareMathOperator{\Proj}{{Proj}}   
\DeclareMathOperator{\rank}{rank}     
\DeclareMathOperator{\Spec}{{Spec}}   
\DeclareMathOperator{\Sym}{{Sym}}     
\DeclareMathOperator{\trop}{{trop}}   
\DeclareMathOperator{\Var}{{V}}       
\DeclareMathOperator{\nor}{nor}     
\DeclareMathOperator{\Nash}{Nash}     
\DeclareMathOperator{\relint}{relint} 
\DeclareMathOperator{\st}{star} 

\newcommand{\wt}{\widetilde}   
\newcommand{\ol}{\overline}    

\newcommand{\abs}[1]{\left|#1\right|}                 
\newcommand{\set}[1]{\left\{#1\right\}}               
\newcommand{\ideal}[1]{\left\langle#1\right\rangle}   
\newcommand{\mat}[1]{{\left[#1\right]}}               

\newcommand{\into}{\hookrightarrow}
\newcommand{\onto}{\twoheadrightarrow}

\newcommand{\irr}{\mathrm{irr}}     
\newcommand{\sm}{\mathrm{sm}}       
\newcommand{\nsm}{\mathrm{nsm}}     
\newcommand{\Frat}{\mathrm{Frat}}   
\newcommand{\trp}{\intercal}        

\renewcommand{\emptyset}{\varnothing}
\renewcommand{\AA}{\mathds{A}}   
\newcommand{\CC}{\mathds{C}}     
\newcommand{\DD}{\mathds{D}}    
\newcommand{\KK}{\mathds{K}}     
\newcommand{\LL}{{\mathds{L}}}   
\newcommand{\NN}{\mathds{N}}     
\newcommand{\PP}{\mathds{P}}     
\newcommand{\RR}{\mathds{R}}     
\newcommand{\TT}{{\mathds{T}}}   
\newcommand{\ZZ}{{\mathds{Z}}}   

\newcommand{\bF}{{\mathbf{F}}}   
\newcommand{\bG}{{\mathbf{G}}}   
\newcommand{\bS}{{\mathbf{S}}}   
\newcommand{\bT}{{\mathbf{T}}}   

\newcommand{\A}{{\mathscr{A}}}  
\newcommand{\B}{\mathcal{B}}    
\renewcommand{\L}{\mathscr{L}}  

\newcommand{\sF}{\mathscr{F}}   
\newcommand{\sO}{\mathscr{O}}   

\newcommand{\M}{\mathsf{M}}     
\newcommand{\U}{\mathsf{U}}     
\newcommand{\matM}{{\mathsf{M}}}
\newcommand{\matU}{{\mathsf{U}}}

\newcommand{\fraka}{{\mathfrak a}}
\newcommand{\frakb}{{\mathfrak b}}
\newcommand{\frakc}{{\mathfrak c}}
\newcommand{\frakm}{{\mathfrak m}}
\newcommand{\frakp}{{\mathfrak p}}
\newcommand{\frakq}{{\mathfrak q}}

\newcommand{\Ppt}[5]{{(#1\mathbin{:} #2\mathbin{:} #3\mathbin{:} #4\mathbin{:} #5)}}

\tikzset{square matrix/.style={
    matrix of nodes,
    column sep=-\pgflinewidth, row sep=-\pgflinewidth,
    nodes={draw,
      minimum height=0.5cm,
      anchor=center,
      text width=0.6cm,
      align=center,
      inner sep=0pt
    },
  },
  square matrix/.default=0.55cm
}

\title[Tropical resolutions of configuration hypersurfaces]{Tropical resolutions of\\ configuration hypersurfaces}

\author[D.~Bath]{Daniel Bath}
\address{\linebreak
  Daniel Bath\\
  Departement Wiskunde\\
  KU Leuven\\
  3001 Leuven, Belgium
}
\email{\href{mailto:dan.bath@kuleuven.be}{dan.bath@kuleuven.be}}
\thanks{DB was supported by FWO grant \#1282226N}

\author[G.~Denham]{Graham Denham}
\address{\linebreak
  Graham Denham\\
  Department of Mathematics, University of Western Ontario\\
  London, ON\\
  Canada N6A 5B7
}
\email{\href{mailto:gdenham@uwo.ca}{gdenham@uwo.ca}}
\thanks{GD was supported by NSERC of Canada}

\author[M.~Schulze]{Mathias Schulze}
\address{\linebreak
  Mathias Schulze\\
  Department of Mathematics, RPTU University Kaiserslautern-Landau\\ 
  67663 Kaiserslautern\\
  Germany
}
\email{\href{mailto:mschulze@rptu.de}{mschulze@rptu.de}} 
\thanks{}

\author[U.~Walther]{Uli Walther}
\address{\linebreak
  U.~Walther\\
  Department of Mathematics\\
  Purdue University\\
  West Lafayette, IN 47907\\
  USA
}
\email{\href{mailto:walther@purdue.edu}{walther@purdue.edu}}
\thanks{UW was supported by NSF grant DMS-2100288 and by
 Simons Foundation Collaboration Grant for Mathematicians \#580839 and SFI-MPS-TSM-00012928}

\keywords{Configuration, matroid, graph, incidence, determinantal, singularity, resolution, Feynman, Kirchhoff, Symanzik, Nash, sch\"on, conormal fan}

\subjclass[2020]{Primary 14N20, 32S45; Secondary 05C31, 14T20, 81Q30}

\begin{document}

\begin{abstract}
Configuration polynomials generalize the Kirchhoff polynomial of a graph, as well as the Symanzik polynomials that appear in the denominators of Feynman integrands.  The \emph{configuration hypersurfaces} cut out by such polynomials are typically highly singular, which poses a challenge for the evaluation of Feynman integrals even in simplified settings.

In this paper, we provide a two-step recipe for a resolution of singularities of any irreducible configuration hypersurface.  We first consider the normalization of the Nash blow-up, which we identify with an incidence variety introduced by Bloch~\cite{Blo09,Blo20}.  This variety is typically still not smooth, but it is the closure of a smooth subvariety of a torus.  The latter then admits a smooth, tropical compactification, using work of Tevelev.  
We construct explicitly such a compactification and a morphism to the normalized Nash blow-up for every configuration, described in terms of bipermutohedral matroid combinatorics introduced by Ardila, Denham and Huh~\cite{ADH23}.

Along the way, we find that the normalized Nash blow-up of the configuration hypersurface has strongly $F$-regular singularities in positive characteristic.  We deduce this by certifying $F$-rationality of its biprojective cone, and infer from it that the normalized Nash blow-up has rational singularities over the complex numbers. 
\end{abstract}


\maketitle
\setcounter{tocdepth}{2}
\tableofcontents

\section{Introduction}

Feynman diagrams encode particle interactions in high-energy physics, and
Feynman integrals can be used to compute the probabilities of specific
interactions.  Even in simplified settings that neglect mass and momenta,
evaluation of Feynman integrals poses serious challenges due to singularities and issues with convergence.  

The Feynman integrand is the square root of a rational function, the denominator of
which involves the \emph{(first) Symanzik polynomial} $\sum_{T\in\B(G)}
\prod_{e\not\in T}x_e$,
where $\B(G)$ denotes the set of spanning trees in
the Feynman graph $G$.
The convergence properties and evaluation techniques for the Feynman integral,
then, depend on the geometry of the \emph{graph hypersurface} cut out by
the Symanzik polynomial.  This hypersurface is typically highly singular, and
the main focus of this paper is to provide an explicit resolution of
its singularities, expressed entirely in terms of combinatorial data from
the Feynman graph. 

The search for a resolution of singularities for Feynman integrands was initiated by Bloch, Esnault and Kreimer~\cite{BEK06}, and one of their key insights was that the construction of 
Symanzik polynomials generalizes naturally from graphs to realizable matroids. More precisely,  they found that the
first and second Symanzik polynomials are both instances of \emph{configuration
polynomials} $\psi_W$; these are induced by the choice of a subspace $W$ inside a vector space $V$ with distinguished basis---compare Proposition \ref{prop:expand_psi}. The mathematical advantage of the larger world of configuration polynomials  is a greater flexibility; for example there is a general notion of duality that extends that of
planar graphs.  Our methods and results apply to this  more general setting and make use of the added features.

From now on, let $V\coloneqq \bigoplus_{i\in E}\KK\cdot x_i$ denote the vector space spanned by all edge variables  $x_i$ in the Feynman integral, and consider the Nash blow-up of the (projective) configuration hypersurface $X_W\subseteq \PP V$.
This approach was pioneered in \cite[\S4]{BEK06}, where they show it produces a resolution of singularities under a (unstated) hypothesis about a certain map being an embedding; for an elucidation see \cite{PattersonThesis}, and in particular \cite[Prop.~2.2.5, Thm.~4.4.1]{PattersonThesis}.
Alas, this approach rarely provides a resolution of singularities.
Indeed, we see in \S\S\ref{subsec:Lambda_smooth}, \ref{subsec:X_dual} that the Nash blow-up is often not even normal, and it can be  
smooth only when the configuration comes from a \emph{round matroid}.  Since the class
of round matroids intersects more or less vacuously with the configurations
coming from interesting Feynman diagrams, one must go further. 

In subsequent years, Bloch~\cite{Blo09,Blo20} introduced a 
complete intersection variety $\Lambda_W$  
of bihomogeneous quadrics inside $\PP V\times \PP V^\star$ mapping surjectively onto $X_W$.  We show (Theorem~\ref{thm:normalize}) that $\Lambda_W$
is in fact the normalization of the Nash blow-up, giving a commutative diagram
\begin{equation}\label{eq:intro_diagram}
\begin{tikzcd}[row sep=5pt]
  \PP V\times \PP V^\star\ar[r,"="] & \PP V\times \PP V^\star \ar[r,"p_2"] & \PP V^\star\\
  \Lambda_W\ar[twoheadrightarrow,r]\arrow[u, phantom, sloped, "\subseteq"]
  & \Nash(X_W)\ar[twoheadrightarrow,r] \arrow[u, phantom, sloped, "\subseteq"]
  & X_W \arrow[u, phantom, sloped, "\subseteq"]
\end{tikzcd}
\end{equation}
in which the maps in the lower row are birational. 




Unfortunately, the variety $\Lambda_W$ generally does not provide a resolution of singularities; however, the situation has markedly improved.  While $X_W$ can have singularities on the big torus in $\PP V$, the restriction $\Lambda_W^\circ$ of $\Lambda_W$ to the big  torus inside $\PP V\times \PP V^\star$
is always smooth (Proposition~\ref{prop:Lambda_toric}).  This means that, in order to resolve the
singularities of $\Lambda_W$ (and hence of $X_W$), we can make use of  the techniques of 
\emph{tropical compactifications}
developed by Tevelev~\cite{Tev07}  and refined by Hacking~\cite{Hac08}. 
Roughly speaking, one blows up (torus-equivariantly) the product of projective spaces to a new
toric variety $\PP(\Delta)$ while maintaining an isomorphism on the torus.
Under
suitable conditions (which we show to hold in our situation), the
closure $\wt\Lambda_W=\wt\Lambda_W(\Delta)$ of the smooth subvariety
$\Lambda_W^\circ$ in $\PP(\Delta)$ is guaranteed to be smooth as well, and one obtains an
embedded resolution
\[
\begin{tikzcd}[row sep=5pt]
  \PP(\Delta)\ar[twoheadrightarrow,r] & \PP V\times \PP V^\star\\
  \wt\Lambda_W\ar[twoheadrightarrow,r]\arrow[u, phantom, sloped, "\subseteq"]
  & \Lambda_W.\arrow[u, phantom, sloped, "\subseteq"]
\end{tikzcd}
\]
Tevelev's theory states that the components of the boundary
$\wt\Lambda_W\setminus\Lambda_W^\circ$ are restrictions of torus orbits in
$\PP(\Delta)$: as such, their incidence relations are the same as those of
the torus orbits and are governed by combinatorics of the associated matroid.

Tropical compactifications of $\Lambda_W^\circ$ are not unique; rather, they
are indexed by unimodular fan structures supported on its tropicalization.
We show that the tropicalization $\trop(\Lambda_W^\circ)$ is 
isomorphic to the support of the product of the Bergman fans of the
matroids of $W$ and its dual $W^\perp$ 
(Proposition~\ref{prop:trop_Lambda}).  This (linear) isomorphism $\mu$ comes
from a description of $\Lambda_W$, birationally, as the graph of the
Hadamard product of the projective linear spaces $\PP W$ and $\PP W^\perp$
(Proposition~\ref{prop:Lambda_toric}).

At this point, a combinatorialist
might view the problem to be solved, since these fan structures are
well-understood.  The remaining subtlety, however, is that the isomorphism so constructed does not respect the product structure of \eqref{eq:intro_diagram}; thus,
in order to obtain a tropical compactification $\wt\Lambda_W$ together with
a well-defined map to the configuration hypersurface, some additional
refinement is required. For this purpose, we make use of the bipermutohedral
fan $\Sigma_{E,E}$ introduced by Ardila, Denham and Huh~\cite{ADH23},
with which we recall in \S\ref{subsec:fans}.
The \emph{square conormal fan} $\Sigma_{-\M,\M^\perp}$ is the induced
subfan that refines a product of Bergman fans $(-\Sigma_{\M})\times
\Sigma_{\M^\perp}$ (Proposition~\ref{prop:conormal}).
We let $\wt\Delta_{\M}$ denote its image under the isomorphism
$\mu$.  In \S\ref{subsec:tropical_res}, we prove one of our main results:


\begin{thm}\label{thm:trop_res_intro}
  If $W$ is a configuration with connected matroid $\M$, 
  then the fan $\wt\Delta_{\M}$ gives a tropical compactification $\wt\Lambda_W$ such that $\wt\Lambda_W\to \Lambda_W\to X_W$ is a resolution of singularities, that is, a birational surjection with a smooth source.
\end{thm}
\noindent
By the nature of a tropical compactification, the boundary structure of $\wt\Lambda_{\M}$ is described completely in terms of the fan $\Sigma_{-\M,\M^\perp}$.


The normalized Nash blow-up $\Lambda_W$ is in some sense minimal over $X_W$ but often not smooth, whereas our tropical
resolution is smooth, but often not minimal.  That is, the composition above is a
bijection on an open subset of $X_W$ which is, in general, properly contained
in its subscheme of regular points.  As a trade-off, however, the structure
of the resolution is entirely combinatorial, and its structure (in the
case of graph hypersurfaces) can be read directly from the Feynman diagram.

\bigskip

As noted, we show that the smoothness of $\Lambda_W$ is equivalent to the \emph{roundness} of the underlying matroid, which is the condition that all complements of (proper) flats span
the matroid. For graphical configurations, this means that our tropical
techniques for resolving $X_W$ by further resolving $\Lambda_W$ are
required in nearly all cases.  Nonetheless, $\Lambda_W$ always has a number
of very interesting properties, both from the geometric and the
algebraic viewpoint; the rest of the introduction summarizes our findings about $\Lambda_W$ (and related varieties) in and of themselves.  

If the matroid of $W$ is connected,
then $\Lambda_W$ is a rational image of a product of projective linear
spaces, via a Hadamard product construction and, in particular, is
irreducible (Corollary~\ref{cor:Lambda_facts}). Additionally, as
\cite{BEK06} already observed, $\Lambda_W$ is always a complete
intersection in $\PP V \times\PP V^\star$.

In \S\ref{subsec:Lambda_classes}, we observe a combinatorial formula for its motivic class in $K_0({\mathrm{Var}}_{\CC})$, irrespective of roundness:


\begin{thm}\label{thm:motive_intro}
If $\M$ denotes the matroid of a configuration $W$ defined over $\CC$ and
$\LL=[\AA^1]$ is the class of the affine line, then
\[
  [\Lambda_W] 
= \sum_{
    \substack{F\in\L_{\M} \\ F\neq E}}
\overline{\chi}_{\M/F}(\LL)\cdot(\LL^{\abs{E}-\rank(\M\setminus F)}-1)/(\LL-1)
\]
in Grothendieck's ring of varieties. Here, 
$\overline{\chi}_\M(t)$ denotes the reduced
characteristic polynomial of a matroid $\M$ and
$\L_{\M}$ is the lattice of flats of $\M$.
\end{thm}
\noindent In particular, we find that $[\Lambda_W] \in \ZZ[\LL]$ is an \emph{integer} class.
By contrast, $[X_W]$ generally is not an integer class but rather can be an arbitrarily complicated member of the Grothendieck ring of varieties, in a sense that is made precise in \cite{BelkaleBrosnan}.

\smallskip

Prompted by the minimal model program, many singularity classes have crystallized 
$Y$ that can be tested via a resolution of singularities
$\pi\colon \tilde Y\to Y$ in characteristic zero. Two very important
and successful classes of this type are that of \emph{rational
singularities}, requiring the identity $R\pi_*(\sO_{\tilde
  Y})=\sO_Y$, 
  and the weaker 
\emph{du Bois} property.  In characteristic $p>0$, from the seminal work
of Hochster and Huneke on tight closure, notions have emerged that rest on asymptotic numeric data of the behavior of
(iterates of) the Frobenius endomorphism. For example,
\emph{$F$-purity} asks that viewing any module in characteristic $p>0$
through the $p$-power map should preserve monomorphisms. This turns
out to be approximately asking for the Frobenius to be a split
morphism, and a strengthening is \emph{strong $F$-regularity} which
requires this splitting for the composition of the $p$-th power map
(or one of its iterates) with multiplication by an arbitrary
non-zerodivisor. A third property is \emph{$F$-rationality}, slightly
weaker than strong $F$-regularity and based on asymptotics of the
Frobenius on local cohomology.

Astonishingly, work of Smith, Hara and others has shown that these
properties in positive and zero characteristics respectively are very intimately
related; see
\S \ref{subsec-affinecone} for details of this
interplay.  Moreover,  even though the
characteristic $p$ approach requires ostensibly the checking of
infinitely many containments of ideals, in many situations it can
be done successfully in practice,
and  a number of
varieties have been certified as being du Bois or as having rational
singularities by showing the positive characteristic counterparts. Our
last major item in the introduction bears witness to this pattern.

Let $\hat \Lambda_W$ be the (affine) variety in $V\times V^\star$ defined by
the equations of the complete intersection that cut out $\Lambda_W$ in $\PP
V\times \PP V^\star$.  Then $\hat \Lambda_W$ is the union of two
reduced components, $\hat \Lambda_W=(W\times \{0\})\cup \hat
\Lambda_{W,0}$, and $\hat \Lambda_{W,0}$ is the biprojective cone of
$\Lambda_W$.  In Theorem~\ref{thm-cone-Frat}, we show:
\begin{thm} 
The cone $\hat \Lambda_{W,0}$ is $F$-rational in every positive
characteristic when the field is perfect.\footnote{or just $F$-finite}
Consequently,
\begin{enumerate}
\item $\hat \Lambda_{W,0}$ is Cohen--Macaulay and normal in all
  characteristics, and has rational singularities over the complex
  numbers.
\item $\Lambda_W$ is strongly $F$-regular in positive characteristic,
  and has all the properties listed for $\hat \Lambda_{W,0}$.
\end{enumerate}
\end{thm}
Moreover, we find that 
$\hat\Lambda_{W,0}$ has Cohen--Macaulay type equal to the
rank of the matroid, and its coordinate ring has a very simple free
resolution over the coordinate ring of $V\times V^\star$
(Theorem~\ref{thm-hatLambda-eqns}.)

\subsection{Outline}

In Section~\S\ref{sec:defs} we recall some definitions as well as basic results, old and new, on configuration polynomials. Section \S\ref{sec:lambda} revisits Bloch's incidence variety $\Lambda_W$.  We characterize when it is smooth
in \S\ref{subsec:Lambda_round}, and we establish the $F$-rationality of the affine version in Section~\S\ref{subsec-affinecone}.  We investigate the relationship between the Nash blow-up of the
configuration hypersurface and Bloch's variety in \S\ref{subsec:X_dual}, using
the results of the previous section to show that $\Lambda_W$ is normal.
In Section \S\ref{sec:tropical},  we develop a resolution of singularities for Bloch's variety, hence for the configuration hypersurface, by tropical methods.

\subsection{Assumptions and reference texts}

If the denominator in a Feynman diagram factors, one can decompose the Feynman integral into factors as well. This happens precisely when the underlying matroid is not connected. Our theorems thus focus on the connected case, although in some of our constructions we do need to pass through disconnected matroids.
Similarly, the basic definitions can be made over any field, but we often require geometric scenarios.  For simplicity, we will assume throughout that our base field $\KK$ is always perfect.  Further assumptions about the field or matroid are indicated at the at the start of  sections/subsections when stronger hypotheses are appropriate.  

Our reference for matroid theory is the book of Oxley~\cite{Oxl11}.
For toric geometry we refer the reader to the book of Cox, Little and Schenck \cite{CLS11}, and for tropical geometry to that of Maclagan and Sturmfels~\cite{MS15}. For matters on the Frobenius morphism we refer to the notes by Ma and Polstra \cite{MaPolstra}.

\subsection{Acknowledgements}

The authors thank the Centro Ennio de Giorgi for its hospitality, where this
project started.  The first three authors also thank the Max Planck
Institute for Mathematics (Bonn) for its support during its development.

\section{Matroids, configurations, and polynomials}\label{sec:defs}

Our starting point is the \emph{configuration polynomial} $\psi_W$, which was
introduced by Bloch, Esnault and Kreimer~\cite{BEK06} and includes the Symanzik
and Kirchhoff polynomials as special cases.  We recall its definition and establish our notation (see \cite[\S2]{DPSW22} for further details).

\subsection{Configurations}

We fix a base field $\KK$ and a finite set $E=\set{1,\dots,n}$.
We denote by 
\[
V\coloneqq\KK^{E}\coloneqq\bigoplus_{i\in E}\KK\cdot x_i
\]
the $n$-dimensional $\KK$-vector space with basis $x=x_1,\dots,x_n$. 
For any subset $F\subseteq E$ the vector space $\KK^F$ with basis $x_F=(x_i)_{i\in F}$ can be realized either as a sub- or a quotient space of $V$.
The assignment $F\mapsto\KK^{F}$ defines two functors from the power set of $E$, covariantly to the category of monomorphisms, and contravariantly to that of epimorphisms of $\KK$-vector spaces.
In particular, we denote the coordinate projection associated to $F\subseteq E$ by
\[
\pi_F\colon \KK^E\onto\KK^F.
\]
A \emph{configuration} is a $\KK$-linear subspace $W\subseteq V$.
We note the inclusion map by
\[
\ell=(\ell_1,\dots,\ell_n)\colon W\into V.
\]
Then $W$ is a $\KK$-linear realization of an underlying \emph{matroid} $\M\coloneqq\M_W$ with set of bases
\[
\B_\M\coloneqq\set{B\subseteq E\mid\pi_B\text{ restricts to an isomorphism }W\to\KK^B}
\]
In this way, any matroid property of $\M_W$ can be considered as a property of a configuration.
In particular, we refer to $W$ as \emph{connected} if $\M_W$ is so.
We often abbreviate the \emph{rank} of $W$ by
\[
r\coloneqq\dim_\KK W=\rank\M_W.
\]
In terms of a basis $w^1,\dots,w^r$ of $W$, the map $\ell$ is given by the transpose $A^\trp$ of the $r\times n$-matrix
\begin{eqnarray}\label{eqn-A-and-w}
A=(w^i_j)\in\KK^{r\times n},\quad w^i_j\coloneqq\ell_j(w^i).
\end{eqnarray}
The bases of $\M_W$ index the maximal subsets of linearly independent columns of $A$.
For $F\subseteq E$, the matroid \emph{deletion} $\M_W\backslash F$ and \emph{contraction} $\M_W/F$ are realized by the corresponding configurations
\begin{eqnarray}\label{eqn-W/F}
W\backslash F\coloneqq\pi_{E\backslash F}(W)\subseteq\KK^{E \backslash F} \quad\text{and}\quad
W/F\coloneqq W\cap \KK^{E\backslash F}\subseteq\KK^{E\backslash F}.
\end{eqnarray}

Denote by $V^\star\coloneqq\Hom_\KK(V,\KK)$ the dual $\KK$-vector space with distinguished basis $y=y_1,\dots,y_n$ dual to $x=x_1,\dots,x_n$.
For $\beta\in V^\star$, we write 
\[
\beta_i\coloneqq y_i(\beta)=\beta(x_i),\quad\beta_F\coloneqq\beta\vert_{\KK^F}
\]
for $i\in E$ and $F\subseteq E$.
Note that $\ell_i$ is the restriction of $y_i$ to $W$.
The \emph{dual configuration} to $W$ is the subspace 
\[
W^\perp\coloneqq(V/W)^\star\subseteq V^\star=\bigoplus_{i\in E}\KK\cdot y_i
\]
with inclusion map $\ell^\perp=(\ell^\perp_1,\dots,\ell^\perp_n)$ realizing the dual matroid $\M_{W^\perp}=(\M_W)^\perp$.
Using duality $V\cong V^{\star\star}$ we may view $x_1,\dots,x_n$ as a basis of $V^{\star\star}$ and $\ell^\perp_i$ as the restriction of $x_i$ to $W^\perp$. 

Finally, for vectors $v\in V$ (or $V^\star$, or their projectivizations), let
\begin{equation}\label{eq:F(v)}
F(v)\coloneqq\set{i\in E\mid v_i=0}.
\end{equation}
For a configuration $W$, the collection of subsets $\L_{\M}:=\set{F(w)\mid w\in W}$
is the set of \emph{flats} of the matroid $\M_W$. 

\subsection{Round matroids}\label{ss:round}

The following type of matroid considered by Geelen, Gerards and Whittle (see \cite{GGW03}) will play a crucial role.


\begin{dfn}\label{def:Msmooth}
A matroid $\matM$ is said to be \emph{round} if every cocircuit (a minimal set that meets every basis of $\matM$) spans $\matM$.
\end{dfn}


\begin{prp}\label{prop:combin_smooth}
A matroid $\M$ on $E$ is round if and only if $\rank(\M\backslash F)=\rank\M$ for all proper flats $F$.
Equivalently, there is no partition $E=E_1\sqcup E_2$ for which neither $E_1$ nor $E_2$ span.
In particular, round matroids are connected.
\end{prp}

\begin{proof}
The  maximal proper flats are precisely the complements of the cocircuits, since no basis can be contained in a proper flat (see \cite[2.1.6.(iii)]{Oxl11}).  
Any flat $F$ is an intersection of maximal proper flats;
see \cite[Prop.~1.7.8]{Oxl11}.
The complement $E\backslash F$ of $F$ is thus a union of cocircuits.

If $\M$ is round and $F$ proper, it follows that $E\backslash F$ spans and $\rank_{\M}(E\backslash F)=\rank_{\M} E$.
Conversely, let $C$ be a cocircuit: then  $C$ is the complement of a maximal proper flat $H$, so the rank condition implies that $C$ spans in $\M$.

This proves the equivalence and the remaining claims are obvious.
\end{proof}


\begin{rmk}\label{rem:combin_smooth}
A non-trivial partition as in Proposition~\ref{prop:combin_smooth} is \enquote{nontrivial on loops} in the terminology of Bloch (see \cite[\S4]{Blo20}).
Based on this partition interpretation of roundness, Kung~\cite{Kun86} calls such matroids non-split.
Borissova~\cite{Bor16} characterized regular round matroids as just those of complete graphs $K_n$ for $n\ge2$, together with the dual matroid of the complete bipartite graph $K_{3,3}$.
\end{rmk}


\begin{exa}\label{ex:uniform}
Since a flat of rank $r-1$ has at least $r-1$ elements, if $\M$ is round, we see it is necessary to have $n\geq 2r-1$. 
The proper flats of the uniform matroid $U_{r,n}$ are subsets of $[n]$ of size at most $r-1$, so for uniform matroids, $n\geq 2r-1$ is also a sufficient condition to be round.
\end{exa}

\subsection{Tori and arrangements}

Each coordinate function $y_i$ on $V$ or $x_i$ on $V^\star$ defines a coordinate hyperplane $\Var(y_i)$ or $\Var(x_i)$.
The intersection of their respective complements are the \emph{coordinate tori}
\[
\hat\TT_E\coloneqq V\setminus\Var(y_1\cdots y_n)
\quad\text{ and }\quad
\hat\TT_{E^\star}\coloneqq V^\star\setminus \Var(x_1\cdots x_n).
\]
We let $\TT_E$ and $\TT_{E^\star}$ denote their respective images in $\PP V$ and $\PP V^\star$.
Note that $i\in E$ is not a loop on $\M_W$ if and only if $\ell_i\ne0$.
In this case, $\Var(\ell_i)=W\cap\Var(y_i)$ is a hyperplane in $W$.
Together they form the  \emph{hyperplane arrangements} 
\[
\A_W\coloneqq\set{\Var(\ell_i)\mid i\in E,\ell_i\ne0}
\quad\text{ and }\quad
\A_{W^\perp}\coloneqq\set{\Var(\ell_i^\perp)\mid i\in E,\ell_i^\perp\ne0}
\]
associated with the configurations $W$ and $W^\perp$, respectively.
Their respective complements in $W$ and $W^\perp$ are the intersections
\[
W^\circ\coloneqq W\cap\hat\TT_E\quad \text{ and }\quad(W^\perp)^\circ\coloneqq W^\perp\cap\hat\TT_{E^\star}. 
\]
Their projective counterparts are 
\[
(\PP W)^\circ \coloneqq \PP(W^\circ)=(\PP W)\cap \TT_E \quad\text{and}\quad
(\PP W^\perp)^\circ \coloneqq \PP((W^\perp)^\circ)=\PP (W^\perp)\cap \TT_{E^\star}. 
\]

The torus $\hat\TT_E$ acts on $V$ by linear transformations and thus on $V^\star$ by the \emph{contragredient action}.
For $t\in\hat\TT_E$, $\beta\in V^\star$ and $v\in V$ it is given by
\begin{equation}\label{eq:contra}
(t\beta)(v)=\beta(t^{-1}v)
\end{equation}
and restricts to an action on $\hat\TT_{E^\star}$.

\subsection{Bilinear forms}\label{subsec:bilinear}

The multiplication map $\hat\TT_E\times \hat\TT_E\to\hat\TT_E$ extends to
a $\KK$-bilinear $\hat\TT_E$-biequivariant map, the \emph{Hadamard product}
\begin{equation}\label{eqn-hadamard}
Q_E\colon V\times V\to V,\quad (u,v)=\bigg(\sum_{i \in E} u_i x_i , \sum_{i \in E} v_i x_i \bigg)\mapsto\sum_{i\in E}u_iv_ix_i.
\end{equation}
Considering the first argument as a parameter, dualizing yields a bilinear map
\[
Q_E^\star\colon V\times V^\star\to V^\star,\quad (u,\beta)\mapsto(v \mapsto \beta\circ Q_E(u,v)).
\]


\begin{lem}\label{lem:Q}
Restricted to $\hat\TT_E$ in the first argument, $Q_E^\star$ becomes the inverse contragredient action, that is, $Q_E^\star(u,\beta)=u^{-1}\beta$ for $u\in\hat\TT_E$ and $\beta\in V^\star$.
\end{lem}

\begin{proof}
For $u\in\hat\TT_E$, $\beta\in V^\star$ and $v\in V$, we have
\[
Q_E^\star(u,\beta)(v)=\beta\circ Q_E(u,v)=\beta(uv)=(u^{-1}\beta)(v).\qedhere
\]
\end{proof}


Composition of $Q_E$ and $Q^\star_E$ with $W\into V$ and $V^\star\onto W^\star$ leads to bilinear maps
\[
Q_W\colon W\times W\to V,\quad Q_W^\star\colon W\times V^\star\to W^\star.
\]
Considering the first argument as a parameter, $Q_W^\star$ becomes the dual of $Q_W$, that is, 
\[
Q_W^\star(w,\beta)=\beta\circ Q_W(w,-).
\]
We use square brackets $\mat{-}$ to denote matrices of (bi)linear maps with respect to the coordinate basis of $V$, a chosen basis of $W$ and their respective dual bases.
Then $[\ell]=A^\trp$, $[Q_E]=\diag(x_1,\dots,x_n)$, $\mat{Q_E^\star(-,\beta)}=\diag([\beta])$ and
\begin{equation}\label{eq:matrices}
\mat{Q_W}= A\mat{Q_E} A^\trp,\quad
\beta(\mat{Q_W})=
\mat{\beta\circ Q_W}=
\mat{Q_W^\star(-,\beta)}=
A\mat{Q_E^\star(-,\beta)} A^\trp.
\end{equation}

\subsection{Configuration polynomials}

Consider a configuration $W\subseteq V=\KK^E$ of rank $r$.
Its \emph{configuration polynomial} is the homogeneous polynomial of degree $r$ defined (up to a nonzero square in $\KK$ by using a basis of $W$ for the determinant) as
\[
\psi_W\coloneqq\det(Q_W)\in\Sym(V)=\KK[x_1,\dots,x_n].
\]
To avoid triviality, we only consider configuration polynomials for which
$n>r>0$.


\begin{prp}[{\cite[Cor.\ 1.4]{BEK06}}]\label{prop:expand_psi}
The configuration polynomial of a configuration $W\subseteq V$ is squarefree and can be written as 
\[
\psi_W = \sum_{B\in \B_{\M_W}}\det(\pi_B)^2\prod_{i\in B}x_i.
\]
\end{prp}


\begin{rmk}

The first and second Symanzik polynomials of a graph are configuration polynomials (see \cite[\S2]{BEK06} and \cite[Def.\ 3.6]{Pat10}), cf.~Example~\ref{ex:delA3}.
Moreover, the following conditions are equivalent: a)  $W$ is contained in the coordinate hyperplane $\Var(x_i=0)$; b)  $\psi_W$ does not involve $x_i$ explicitly; c) $i$ is a loop of $\matM_W$; d) $\ell_i=0$. 
\end{rmk}


The configuration polynomial defines an affine/projective \emph{configuration hypersurface}
\[
\hat X_W\coloneqq\Var(\psi_W)=\set{\beta\in V^\star\mid\rank(\beta\circ Q_W)<r}\subseteq V^\star,\quad
X_W\coloneqq\PP\hat X_W\subseteq\PP V^\star
\]
with torus part $X_W^\circ\coloneqq X_W\cap\TT_{E^\star}$. 

\smallskip

Recall that a \emph{very affine variety} is a closed irreducible subvariety of an algebraic torus.


\begin{prp}[{\cite[Prop.~3.8]{DSW21}}]\label{prop:X_very_affine}
If $W\subseteq V$ is a connected configuration, 
then $X_W^\circ\subseteq \TT_{E^\star}$ is a very affine variety with closure $\overline{X_W^\circ}=X_W$.
\end{prp}


\begin{thm}[{\cite[Thm.~4.1]{Pat10}}, {\cite[Main Thm.]{DSW21}}]\label{thm:X_sing}
Let $W\subseteq V$ be a configuration with matroid $\M=\M_W$
Then the non-smooth loci of the configuration hypersurfaces are given by
\[
\hat X_W^\nsm=\set{\beta\in V^\star\mid\rank(\beta\circ Q_W)<\rank\M-1},\quad
X_W^\nsm=\PP X_W^\nsm.
\]
If $\M$ is connected and $\rank\M\ge 2$, then the codimensions in $V^\star$ and $\PP V^\star$ equal $3$.\qed
\end{thm}

\section{Bloch's incidence variety}\label{sec:lambda}

A first step towards constructing an explicit resolution of singularities of $X_W$ appeared in \cite[\S4]{BEK06} and later in work of Bloch~\cite{Blo09,Blo20}.
We recall the latter in \S~\ref{subsec:Lambda} and compare with the former in \S~\ref{subsec:X_dual}. 

\subsection{Bloch's definitions}\label{subsec:Lambda}

Bloch associated to a configuration $W\subseteq V$ an incidence subvariety $\Lambda_W$ of $\PP^{r-1}\times\PP V^\star$, where $\PP^{r-1}\cong\PP W$ by a choice of basis.


\begin{dfn}\label{def:Lambda}
For a configuration $W\subseteq V$, consider the affine variety
\[
\hat\Lambda_W\coloneqq\set{(w,\beta)\in W\times V^\star\mid Q_W^\star(w,\beta)=0}\xhookrightarrow{\ell\times\id}V\times V^\star.
\]
The biprojectivization of $\set{(w,\beta)\in\hat\Lambda_W\mid w,\beta\neq 0}$ defines the variety 
\[
\Lambda_W\subseteq\PP W\times X_W\subseteq\PP W\times \PP V^\star
\xhookrightarrow{\PP\ell\times\id}\PP V\times \PP V^\star.
\]

We will without notice move freely between $\hat\Lambda_W$ and $\Lambda_W$ on one side, 
 and their images under the inclusions induced by 
$\ell$ on the other.
\end{dfn}

Using the matrices \eqref{eq:matrices}, both $\hat\Lambda_W$ and $\Lambda_W$ are defined explicitly by the $r$ equations 
\begin{align}\label{eq:Lambda_mat2}
  \hat\Lambda_W&=\set{(w,\beta)\mid A D_{[\beta]} A^\trp [w]=0}\\
  &=\set{(w,\beta)\mid A D_{A^\trp [w]}[\beta]=0}\nonumber, 
\end{align}
where $D_{[v]}$ denotes the $n\times n$ diagonal matrix with entries $(D_v)_{i,i}=v_i$.


Bloch also gave an alternative more abstract description of $\Lambda_W$ as a scheme.
Consider the coherent sheaf $\sF$ on $\PP W$ defined by the presentation
\begin{equation}\label{eq:Lambda_sheaf}
\begin{tikzcd}[column sep=small]
0\ar{r} & \sO_{\PP W}\otimes_{\KK} W \ar{r}{\varphi} & \sO_{\PP W}(1)\otimes_{\KK} V \ar[r] & \sF\ar[r] & 0,
\end{tikzcd}
\end{equation}
where
\[
\varphi(1\otimes w)=Q_W(w,-)=\sum_{i\in E}w_i\ell_i\otimes x_i
\in W^\star\otimes V
\]
The injectivity of $\varphi$ is due to the fact that the sections $\ell_i\otimes x_i$ of $\sO_{\PP W}(1)\otimes_{\KK} V$ are locally $\sO_{\PP W}$-linearly independent.

\begin{prp}[{\cite[Prop.~3.5.(i)]{Blo20}}]\label{prop:Lambda_ProjSym}
For any configuration $W\subseteq V$, there is an isomorphism of schemes $\Lambda_W\cong\Proj\Sym(\sF)$.
\noproof
\end{prp}

Bloch deduces various properties of $\Lambda_W$ using the two projections
\begin{equation}\label{eq:projections}
\begin{tikzcd}[row sep=small]
\PP V & \PP V\times\PP V^\star\ar{l}[swap]{p_1}\ar{r}{p_2} & \PP V^\star\\
\PP W\ar[u,phantom,sloped,"\subseteq"] & \PP W\times X_W\ar[u,phantom,sloped,"\subseteq"]\ar{l}[swap]{p_1}\ar{l}[swap]{p_1}\ar{r}{p_2} & X_W\ar[u,phantom,sloped,"\subseteq"]
\end{tikzcd}
\end{equation}

\subsection{Stratification by flats and irreducibility}\label{subsec:Lambda_flats}

The first projection in \eqref{eq:projections} induces a stratification of $\Lambda_W$ by flats.
For any flat $F$ of $\M_W$, consider the (affine) strata in $\PP W$ and $\Lambda_W$ indexed by $F$ (cf.~\eqref{eqn-W/F}):
\[
U_F\coloneqq\PP(W/F)^\circ,\qquad \Lambda_W\vert_{U_F}\coloneqq\Lambda_W\cap p_1^{-1}(U_F).
\]
Note that $0\ne w\in W$ has projective image in $U_F$ exactly if $F=F(w)$ (see \eqref{eq:F(v)}).


\begin{lem}\label{lem:nonempty_strata}
Let $W\subseteq V$ be a configuration.
For any flat $F$ of $\M_W$ and $\abs\KK\gg0$, $U_F\ne\emptyset$ exactly if $F$ is proper.
In this case, $\dim U_F=\rank\M-\rank(F)-1$.
\end{lem}

\begin{proof}
Note that $U_E=\emptyset$.
If $F$ is proper, then $\M_W/F$ contains no loops (see \cite[\S3.1, Ex.\ 8.(a)]{Oxl11}).
This means that $\ell_i\ne0$ for all $i\in E\backslash F$ and hence $U_F\ne\emptyset$ for $\abs\KK\gg0$. 
The dimension statement follows from the definition of $U_F$.
\end{proof}


As a direct consequence of the description of $\Lambda_W$ in \eqref{eq:Lambda_mat2} we record the following lower bound.

\begin{lem}\label{lem:Lambda_bound}
For any flat $F$ of $\M_W$, the product $U_F \times V(x_i\mid i \in E \backslash F)$ is contained in $\Lambda_W$.
In particular, $\Lambda_W\vert_{U_F}$ is nonempty if $U_F$ is nonempty and $F\ne\emptyset$.
\end{lem}

\begin{proof}
For any $w\in W$ with projective image in $U_F$, we have $(A^\trp [w])_j=0$ if and only if $j\in F$. 
For such a $w$ and $\beta\in V(x_i\mid i \in E \backslash F)$ then $D_{A^\trp [w]}[\beta]=0$.
The claimed containment thus follows from \eqref{eq:Lambda_mat2}.
\end{proof}


Bloch~\cite[\S4]{Blo20} shows that $\sF$ is a vector bundle on the strata defined by flats.


\begin{prp}\label{prop:Lambda_strata}
For any configuration $W\subseteq V$ and any flat $F$ of $\M=\M_W$, there is an isomorphism of schemes
\[
\Lambda_W\vert_{U_F}\cong U_F\times\PP^{n_F-1}, \qquad n_F\coloneqq\abs E-\rank(\M\backslash F).
\]
If $\M$ is round, then $\Lambda_W$ is a bundle over $\PP W$ whose fibre is $\PP^{\abs{E}-\rank\M-1}$.
\end{prp}

\begin{proof}
Choose a basis $B=\set{i_1,\dots,i_m}$ of $\M\backslash F$ and vectors $w^1,\dots,w^m\in W$ that project under $\pi_{E\backslash F}$ to a basis of $W\backslash F$ such that $w^j_{i_{k}}=\delta_{j,k}$.
Note that $U_F$ lies in the affine chart of $\PP(W/F)$ defined by $\ell_k=1$, where $k\in E\backslash F$.
This identifies
\[
\sO_{U_F}(1)\cong\sO_{U_F},\quad\ell_i\mapsto\frac{\ell_i}{\ell_k}.
\]
Since $\frac{\ell_i}{\ell_k}\vert_{U_F}=0$ for $i\in F$, we find that
\[
\varphi(1\otimes w)\vert_{U_F}=\sum_{i\in E\backslash F}w_i\frac{\ell_i}{\ell_k}\vert_{U_F}\otimes x_i.
\]
In particular, $\varphi(1\otimes (W\cap\KK^F))\vert_{U_F}=0$ where $\KK^F=\ker\pi_{E\backslash F}$.
Restricted to $U_F$, the image of $\varphi$ is thus generated of $\sO_{U_F}$ by 
\[
\varphi(1\otimes w^j)\vert_{U_F}\equiv\frac{\ell_{i_j}}{\ell_k}\otimes x_{i_j}\mod\ideal{1\otimes x_i\mid i\in E\backslash B},\quad j=1,\dots,m.
\]
Since $\frac{\ell_i}{\ell_k}\in\KK[U_F]^\times$ is a unit for each $i\in E\backslash F$, this shows that $\sF\vert_{U_F}\cong\sO_{U_F}^{n_F}$.
Applying $\Proj\Sym$ yields $\Lambda_W\vert_{U_F}\cong\PP_{U_F}^{n_F-1}=U_F\times\PP^{n_F-1}$ as claimed.

Suppose now that $\M$ is round and hence $n_F=\abs E-\rank\M$ for all flats $F$.
Then $B$ and $w^1,\dots,w^m\in W$ are bases of, respectively, $\M$ and $W$, and $\frac{\ell_i}{\ell_k}\in\sO_{\PP W,w}^\times$ is a unit for each $i\in B$ and $w\in U_F$.
As above, we find that $\sF_w\cong\sO_{\PP W,w}^{n_F}$ for all $w\in\PP W$.
It follows that $\Lambda_W$ is a projective bundle of rank $n_F$ as claimed.
\end{proof}


Bloch~\cite[Prop.~3.5.(ii)]{Blo20} shows that $\Lambda_W$ is an irreducible complete intersection.


\begin{cor}\label{cor:Lambda_facts}
  Let $W\subseteq V$ be a configuration with matroid $\M=\M_W$.
  Then $\Lambda_W$ is a complete intersection of codimension $\rank\M$ in $\PP W\times\PP V^\star$.
If $\M$ is connected and $\abs\KK\gg0$, then $\Lambda_W$ is an integral scheme, the closure of the open stratum $\Lambda_W\vert_{U_\emptyset}$.
Conversely, if $\Lambda_W$ is irreducible and $\M$ is loopless, then $\M$ must be connected.
\end{cor}

\begin{proof}
By Lemma~\ref{lem:nonempty_strata} and Proposition~\ref{prop:Lambda_strata} the stratum of $\Lambda_W$ indexed by any proper flat $F$ of $\M$ has dimension
\begin{align*}
  \dim(\Lambda_W\vert_{U_F})&=\dim(W/F)-1+\abs E-\rank(\M\backslash F)-1  \\
  & = \abs E+\rank(\M/F)-\rank(\M\backslash F)-2.
\end{align*}
Its codimension in $\PP W\times\PP V^\star$ therefore equals (see \cite[Prop.~3.1.6]{Oxl11})
\begin{align}
\codim_{\PP W\times\PP V^\star}(\Lambda_W\vert_{U_F})&=\rank\M-\rank(\M/F)+\rank(\M\backslash F)\nonumber\\
&=\rank(F)+\rank(\M\backslash F)\ge\rank\M.\label{eq:Lambda_facts}
\end{align}
In particular, $\codim_{\PP W\times\PP V^\star}\Lambda_W\ge\rank\M$. 
Since $\Lambda_W$ can be defined by $\rank\M$ equations, it follows
that it is a complete intersection of codimension $\rank\M$.
In particular, it is equidimensional and has no embedded components.

Again by Lemma~\ref{lem:nonempty_strata} and Proposition~\ref{prop:Lambda_strata}, the open stratum $\Lambda_W\vert_{U_\emptyset}$ indexed by the empty flat is nonempty, reduced and irreducible.
Its closure is thus a reduced irreducible component of $\Lambda_W$.
If $\Lambda_W$ is irreducible, it is therefore reduced.

Irreducibility of $\Lambda_W$ occurs exactly if the other strata do not contribute additional components, which is to say that the inequality \eqref{eq:Lambda_facts} be strict for all nonempty proper flats. 
This strictness is equivalent to no such flat being a separator of $\M$ (see \cite[Prop.~4.2.1]{Oxl11}).
Thus, $\Lambda_W$ is irreducible if and only if $\matM$ only allows the trivial separators $\emptyset$ and $\matM$, which is equivalent to $\matM$ being connected since separators are by definition unions of components.

For the converse, note that the complement of a loop is a separator but not a flat, whereas all separators of a loopless matroid are flats (see \cite[\S4.1, Ex.~2]{Oxl11}).
\end{proof}

\subsection{Normality and smoothness}\label{subsec:Lambda_smooth}

We now describe the Jacobian matrix of Bloch's incidence variety and find smooth points over the strata by flats.
These are then used to establish normality using Serre's criterion and to relate smoothness to roundness of the matroid.


\begin{lem}\label{lem:Lambda_Jacobian}
Let $W\subseteq V$ be a configuration with matroid $\M=\M_W$.
Then the (transposed) Jacobian matrix $J_W(w, \beta)$ at $(w,\beta)\in W\times V^\star$ obtained from the equations for $\hat{\Lambda}_W$ in \eqref{eq:Lambda_mat2} reads
\[
J_W(w, \beta)=
\left(\begin{array}{c|c}
A D_{[\beta]} A^\trp & A D_{A^\trp [w]} 
\end{array}\right).
\]
Its column space equals that of 
\[
\left(\begin{array}{c|c}
\sum_{i\in F} \beta_i[\ell_i][\ell_i]^\trp & A_{E\backslash F}
\end{array}\right)
=
\left(\begin{array}{c|c}
[\beta_F \circ Q_{W\vert_F}] & A_{E\backslash F}
\end{array}\right),
\]
where $F=F(w)$ and the matrix $A_{E\backslash F}$ is obtained from $A$ by removing the columns indexed by $F$. 
In particular, 
\[
\rank(\M\backslash F)\le\rank J_W(w,\beta)\le\rank({\set{i\in F\mid \beta_i\ne0}\cup E\backslash F}),
\]
and when $\beta_F\not\in\hat X_{W\vert_F}$  the matrix $J_W(w,\beta)$ has full rank: $\rank J_W(w,\beta)=\rank\M$.
\end{lem}

\begin{proof}
The block matrix expression for $J_W(w,\beta)$ is immediate from \eqref{eq:Lambda_mat2}.
Note that 
\[
A D_{[\beta]} A^\trp=\sum_{i\in E}\beta_i[\ell_i][\ell_i]^\trp
\]
is a sum of square matrices of size $\rank\M$. 
The $j$th column of $\beta_i[\ell_i][\ell_i]^\trp$ is a scaling of $[\ell_i]$ by $\beta_i [\ell_i]_j$.  If $F=F(w)$, then $(A^\trp [w])_j=\ell_j(w)=0$ if and only if $j\in F$ (by \eqref{eq:F(v)}).  That is, the nonzero columns of
$AD_{A^\trp [w]}$ are indexed by $E \backslash F$ are nonzero and agree with the corresponding ones of $A$ up to $\KK^\times$-rescaling.
It follows that the column space of the second block of $J_W(w,\beta)$ agrees with that of $A_{E\backslash F}$.
Using column operations, the summands with indices in $E\backslash F$ in the first block can then be eliminated.
This gives the equality
\[
\sum_{i\in F} \beta_i[\ell_i][\ell_i]^\trp = [\beta_F \circ Q_{W\vert_F}],
\]
and the claimed inequalities follow.
For all $\beta_F\not\in\hat X_{W\vert_F}$, the latter matrix has (full) rank equal to $\rank\M\vert_F=\rank A_F$,
so $J_W(w,\beta)$ does too.
\end{proof}


\begin{lem}\label{lem:Lambda_smooth}
Let $W\subseteq V$ be a configuration with matroid $\M=\M_W$.
Then $\Lambda_W$ has smooth points over $U_F$, that is, $\Lambda_W^\sm\vert_{U_F}\neq\emptyset$, for any proper flat $F\ne\emptyset$ of $\M$ if $\abs{\KK}\gg0$.
\end{lem}

\begin{proof}
Lemma~\ref{lem:nonempty_strata} yields a $w \in W$ with projective image in $U_F$.
Since $F\ne\emptyset$ and $\abs\KK\gg0$, there is an element $\beta'\not\in\hat X_{W\vert_F}$ in the complement of $\hat X_{W\vert_F}$.
Define $\beta\in\PP V^\star$ by $\beta_F\coloneqq\beta'$ and $\beta_{E\backslash F}\coloneqq0$.
Due to Lemmas~\ref{lem:Lambda_bound} and \ref{lem:Lambda_Jacobian}, then $(w, \beta)\in\hat\Lambda_W$ and $J_W(w,\beta)$ has full rank.
It follows that $(w,\beta)$ is a smooth point of $\hat\Lambda_W$.
Its projective image in $\PP W\times\PP V^\star$ is then a smooth point of $\Lambda_W$ over $U_F$.
\end{proof}


\begin{cor}\label{cor:normality}
Let $W\subseteq V$ be a configuration with connected matroid $\M=\M_W$. 
Then $\Lambda_W$ is an integral normal scheme.
\end{cor}

\begin{proof}
By hypotheses on $\M$ and Corollary~\ref{cor:Lambda_facts}, $\Lambda_W$ is a reduced complete intersection integral scheme.
By using Serre's criteria for normality and reducedness (applied in affine charts), it is enough to check that $\Lambda_W$ is regular in codimension one.  
By way of contradiction, suppose that $\Lambda_W$ has an irreducible subvariety $Z$ of codimension one with generic point in $\Lambda_W^\nsm$.
By Proposition~\ref{prop:Lambda_strata}, $\Lambda_W\vert_{U_\emptyset}$ is smooth and hence $Z\subseteq\ol{\Lambda_W\vert_{U_F}}$ for some proper flat $F\ne\emptyset$, where $\Lambda_W\vert_{U_F}$ is irreducible.
The strict inequality \eqref{eq:Lambda_facts} in the proof of Corollary~\ref{cor:Lambda_facts} shows that $\Lambda_W\vert_{U_F}$ has codimension at least one.
It follows that $Z=\ol{\Lambda_W\vert_{U_F}}$ and then 
\[
\Lambda^\sm_W\supseteq\Lambda^\sm_W\vert_{U_F}\subseteq\Lambda_W\vert_{U_F}\subseteq\ol{\Lambda_W\vert_{U_F}}=Z\subseteq\Lambda_W^{\nsm}=\Lambda\setminus\Lambda_W^{\sm}
\]
implies $\Lambda^\sm_W\vert_{U_F}=\emptyset$, which contradicts Lemma~\ref{lem:Lambda_smooth} if $\abs{\KK}\gg0$.

Finally, by base change, we may assume that $\KK$ is infinite, since 
normality descends along faithfully flat ring maps (see  \cite[\href{https://stacks.math.columbia.edu/tag/033G}{Tag 033G}]{stacks}).
\end{proof}


Proposition~\ref{prop:Lambda_strata} shows in particular that $\Lambda_W$ is smooth if $\M_W$ is round.
We now establish a refined converse statement: the non-smooth strata arise exactly from flats that violate roundness.


\begin{thm}\label{thm:Lambda_smooth_round}
Let $W\subseteq V$ be a configuration with connected matroid $\M = \M_W$.
For any proper flat $F$ of $\M$ and $\abs\KK\gg0$, $\Lambda_W^{\nsm}\vert_{U_F}\neq\emptyset$ if and only if $\rank(\M \backslash F) < \rank\M$. In particular, $\Lambda_W$ is smooth if and only if $\M$ is round.
\end{thm}

\begin{proof}
Lemma~\ref{lem:nonempty_strata} yields a $w \in W$ with projective image in $U_F$.

First, suppose that $\rank(\M \backslash F) < \rank\M$.
Since $\M$ is connected, $E\backslash F$ is not a flat and hence $\rank(E\backslash F)=\rank(\set{j}\cup(E\backslash F))$ for some $j \in F$.
Define $\beta \in V^\star$ by $\beta_i=\delta_{i,j}$.
By Lemmas~\ref{lem:Lambda_bound} and \ref{lem:Lambda_Jacobian}, $(w,\beta)\in\hat\Lambda_W\vert_{U_F}$ with 
\[
\rank J_W(w,\beta)\le\rank(\set{j}\cup E\backslash F)=\rank(\M\backslash F)<\rank\M.
\]
The projective image of $(w,\beta)$ thus lies in $\Lambda_W^\nsm$.

Now, suppose that $\rank(\M \backslash F) = \rank\M$. 
Represent an arbitrary point of $\Lambda_W\vert_{U_F}$ as the projective image of $(w, \beta)\in\hat\Lambda_W$.
By Lemma~\ref{lem:Lambda_Jacobian}, $\rank J_W(w,\beta)\ge\rank(E\backslash F)=\rank\M$.
The projective image of $(w,\beta)$ thus lies in $\Lambda_W^\sm$.
\end{proof}

\subsection{An example}\label{subsec:Lambda_example}

The following example illustrates the statement of Theorem~\ref{thm:Lambda_smooth_round}.


\begin{exa}\label{ex:delA3}
If $\M=\M_G$ is the matroid of a graph $G$ and $W=W_G$ the configuration resulting from the (pruned) incidence matrix to an arbitrary orientation (see \cite[\S2.4]{DSW21}), we replace the index $W$ by $G$.
For the graph $G$ in Figure~\ref{fig:delA3}, consider the connected graphic matroid $\M=\M_G$ with $n=5$ and $r=3$, and the flats $F_1=\set{1,2,4}$ and $F_2=\set{1,3,5}$.
\begin{figure}[ht]
\begin{tikzpicture}[scale=1.0,baseline=(current bounding box.center),
plain/.style={circle,draw,inner sep=1.5pt,fill=white}]
  \node[plain] (a) at (0,0) {};
  \node[plain] (b) at (1,0) {};
  \node[plain] (c) at (1,1) {};
  \node[plain] (d) at (0,1) {};
  \node (0) at (0.3,0.4) [label=right:1] {};
  \draw (a) edge["2",swap] (b);
  \draw (b) edge["4",swap] (c);
  \draw (c) edge["3",swap] (d);
  \draw (d) edge["5",swap] (a);
  \draw (a) -- (c);
\end{tikzpicture}
\caption{A graph with non-round matroid.}\label{fig:delA3}
\end{figure}
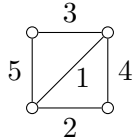
Then $W_G$ is the row span of the matrix
\[
A=
\begin{pmatrix}
1 & 0 & 0 & 1 & 1\\
0 & 1 & 0 & 1 & 0\\
0 & 0 & 1 & 0 & 1\\
\end{pmatrix}.
\]
For $i\in\set{1,2}$, we have $\rank(\M\backslash F_i)<3$ and hence $\M$ is not round.
One computes 
\begin{align*}
[Q_G]&=
\begin{pmatrix}
x_1+x_4+x_5 & x_4 & x_5\\
x_4 & x_2+x_4 & 0\\
x_5 & 0 & x_3+x_5\\
\end{pmatrix}\quad\text{and}\\
\Psi_G&=x_1x_2x_3+x_1x_3x_4+x_2x_3x_4+x_1x_2x_5+x_2x_3x_5+x_1x_4x_5+x_2x_4x_5+x_3x_4x_5.
\end{align*}
Using \cite[Lem.~4.22]{DSW21}, we find a decomposition of the non-smooth locus of $X_G$ into irreducible components
\[
X_G^{\nsm}=C_{F_1}\cup C_{F_2}\quad\text{where}\quad C_F\coloneqq\Var(\Psi_{G_F},\set{x_i\mid i\not\in F})
\]
and $G_F$ denotes the induced subgraph with edges $F$.
For both handles $\set{2,4}$ and $\set{3,5}$, any irreducible component $C$ different from both $C_{F_i}$ for $i\in\set{1,2}$ must be of type (c) in loc.~cit.
But this entails $C\subseteq V(x_2,x_3,x_4,x_5)$ and contradicts the codimension statement in Theorem~\ref{thm:X_sing}.
By symmetry then both components $C_{F_i}$ for $i\in\set{1,2}$ must occur.

Clearly, $C_{F_1}\cap C_{F_2}=\set{\beta_0}$ where $\beta_0\coloneqq\Ppt10000$.
For $\beta\in X_G^\nsm$, the matrix $\beta[Q_G]$ has rank $1$, so $p_2^{-1}(\beta)\cong \PP^1$ (recall diagram \eqref{eq:projections}).
In particular, $p_1\circ p_2^{-1}(\beta_0)$ is the line through
\[
\alpha_{F_1}\coloneqq\Ppt00101
\quad\text{and}\quad
\alpha_{F_2}\coloneqq\Ppt01010, 
\]
and $(\alpha_{F_i},\beta)\in p_2^{-1}(\beta)$ for each $i\in\set{1,2}$ and $\beta\in C_{F_i}$: that is,
$p_1\circ p_2^{-1}(C_{F_i})$ is a pencil of lines through $\alpha_{F_i}$.
We find that $\Lambda^\nsm_G$ consists of the two points $(\alpha_{F_i},\beta_0)$ for $i\in\set{1,2}$.

It is exactly for these flats that the \enquote{round} condition fails: 
$Q_G^\star(w,-)$ has rank $2$ for $w=\alpha_{F_i}$ where $i\in\set{1,2}$, and rank $3$
otherwise.
Taking the kernel for each $w$, we compute for each $i\in\set{1,2}$ that $p_1^{-1}(\alpha_{F_i})=\PP\KK^{F_i}\cong \PP^2$ is the coordinate subspace containing the curve $C_{F_i}$, and $p_1^{-1}(w)\cong\PP^1$ otherwise.
We will revisit this example in \S\ref{subsec:X_dual} and \S\ref{subsec:tropical_res}.
\end{exa}
    
\subsection{A candidate resolution}\label{subsec:Lambda_round}

We now focus on the second projection in \eqref{eq:projections}. 
Bloch~\cite[Prop.~3.5.(ii)]{Blo20} showed that $p_2\colon\Lambda_W\to X_W$ is birational.  More precisely:


\begin{prp}\label{prop:pr2_over_Xsmooth}
Let $W\subseteq V$ be a connected configuration. 
Then the second projection $p_2\colon\Lambda_W\to X_W$ is an isomorphism over the smooth locus $X_W^\sm$.
In particular, it is proper birational, and hence surjective.
\end{prp}

\begin{proof}
By Theorem~\ref{thm:X_sing}, we have $\beta\in X_W^\sm$ if and only if $\beta\circ Q_W$ has corank $1$.
We may thus assume that for some basis $w^1,\dots,w^r$ of $W$ the matrix $(Q_W(w^i,w^j))_{1\le i,j<r}$ has full rank in an open neighborhood $U$ of $\beta$ in $X_W$.
Over $U$, $\Lambda_W$ is then given in terms of homogeneous coordinates $z=z_1,\dots,z_r$ of $\PP W$ by linear equations: 
\begin{equation}\label{eq:section}
\Lambda_W\cap p_2^{-1}(U)=V\Big(\sum_{j<r}Q_W(w^i,w^j)z_i+Q_W(w^r,w^j)z_r\mid j=1,\dots,r-1\Big)
\end{equation}
In particular, it maps to the chart $\set{z_r=1}$ of $\PP W$.
The unique solution 
\[
-(Q_W(w^i,w^j)_{1\le i,j<r})^{-1}(Q_W(w^r,w^j)_{1\le j<r},
\]
of the equations in \eqref{eq:section} for $z_r=1$ defines a morphism $\varphi\colon U\to\set{z_r=1}$ such that $(\varphi,\id_U)\colon U\to\set{z_r=1}\times U$ is an inverse of $p_2$ over $U$.
\end{proof}


We obtain the following consequence of Propositions~\ref{prop:Lambda_strata} and \ref{prop:pr2_over_Xsmooth}.


\begin{cor}\label{cor:Bloch_smooth}
For any configuration $W\subseteq V$ whose matroid is round, the second projection $p_2\colon\Lambda_W\to X_W$ is a resolution of singularities that only modifies the singular points of $X_W$.\qed
\end{cor}

\subsection{Singularities of the affine cone}\label{subsec-affinecone}

Suppose a basis for $W$ has been chosen as in \eqref{eqn-A-and-w}, and introduce coordinate functions $u=u_1,\ldots,u_r$ relative to that basis. 
The  projective incidence variety
$\Lambda_W$ is cut out by the ideal 
\[
I_W\coloneq(q_1,\ldots,q_r)\subseteq \Sym(W^\star)\otimes \Sym(V)=\KK[u,x]
\]
generated by
the $r$ quadrics
\[
q_i\coloneqq(A[Q_E]A^\trp[u])_i\text{ with }1\le i\le r\text{ and }[u]=(u_1,\ldots,u_r)^\trp,
\]
homogeneous separately in $u$ and in $x=x_1,\ldots,x_n$.
By row reduction, and suitably renumerating $E$,  there is an $r\times n$ matrix
$A=(\id_r|B)$ whose row span is also $W$, and this replacement of one matrix with row span $W$ by another  corresponds to a coordinate change purely in $u$.  Observe that
now the $(i,i)$-entry of $A[Q_E]A^\trp$ is $x_i$ plus a
$\KK$-linear combination of $x_{r+1},\ldots,x_n$, while no
off-diagonal entry contains any of $x_1,\ldots,x_r$. In particular,
\[
q_i=x_i u_i+\sum_{j=r+1}^n \tilde u_{i,j}x_j
\]
where $\tilde u_{i,j}$ is a $\KK$-linear combination of $u_1,\ldots,u_r$.

Introduce a term order on $\KK[u,x]$ that refines the lex-order on
$x$. Then the lead term of $q_i$ is $x_iu_i$. In particular, the lead
terms of the $q_i$ are square-free and relatively prime.

The relative primeness implies that these elements form a Gr\"obner
basis under the chosen term order (because any s-pair made from
polynomials with relatively prime lead terms reduces to zero via just
those two polynomials; immediate termination of the Buchberger
algorithm shows completion of the Gr\"obner basis search). This in
turn implies that the initial ideal of $I_W$ is
$(x_1u_1,\ldots,x_ru_r)$. Since the initial ideal is a complete
intersection of height $r$, so is $I_W$ (on the generators
$\{q_i\}$), and in particular $I_W$ is equidimensional. Moreover, since $(x_1u_1,\ldots,x_ru_r)$ is a radical
ideal, so is $I_W$.

\subsubsection{Linkage}

The height $r$ prime $I_{W,u}\coloneqq(u_1,\ldots,u_r)$ is an
associated prime of the (equidimensional) complete intersection
$I_W$ of height $r$. Let
\[
I_{W,0}\coloneqq I_W:I_{W,u}
\]
be the ideal quotient and denote by
$\Var(-)$ the affine schemes attached to ideals in a ring.  Since
$I_W$ is radical, $I_{W,0}$ is the radical ideal to the union 
of the components of 
$\Var(I_W)$ different from $\Var(I_{W,u})$.

Let $I_{W,x}=(x_1,\ldots,x_n)$; 
it follows from Corollary \ref{cor:Lambda_facts}, that
$\Var(I_W)\backslash (\Var(I_{W,u})\cup\Var(I_{W,x}))$ is integral over (the infinite field) $\bar\KK$ and hence over $\KK$ as well. Equidimensionality of $I_W$ implies that no
associated prime of $I_W$ can strictly contain $I_{W,u}$. If $n>r$, no associated prime can be equal to, much less contain, 
$I_{W,x}$. It follows that $I_{W,0}$ is in fact a prime ideal when the matroid is connected.


\begin{dfn}
Let $\hat\Lambda_{W,0}\subseteq W\times V^\star\xhookrightarrow{\,\,\,\ell\,\,\,}V\times V^\star$ be the affine variety defined by $I_{W,0}$. Note that the associated biprojective variety is exactly $\Lambda_W$. 
\end{dfn}


By construction, $I_{W,u}=I_W:I_{W,0}$ and
$I_{W,0}=I_W:I_{W,u}$.   Two pure-codimensional ideals
$\fraka,\frakb$ of equal codimension in a ring $R$ are said to be in \emph{direct linkage} if there is a third $R$-ideal $\frakc$ of the same
codimension with:
$\frakc$ is a complete intersection; $\frakc:\fraka=\frakb,
\frakc:\frakb=\fraka$. Thus, $I_{W,u}$ and $I_{W,0}$ are (directly)
  linked via the complete intersection $I_W$. \emph{Linkage} is the equivalence relation
between ideals generated by direct linkage. The notions go back to
Peskine and Szpiro \cite{PeskineSzpiro-Liaison} and force strong
relations between the three ideals. For example, $R/\fraka$ is
Cohen--Macaulay if and only if $R/\frakb$ is so (see
\cite[Prop.~1.3]{PeskineSzpiro-Liaison}).

For directly linked ideals, $\Hom_{R/\frakc}(R/\fraka,R/\frakc)$
identifies in the obvious way with
$(\frakc:\fraka)/\frakc=\frakb/\frakc$. Suppose that $R/\fraka,
R/\frakb$ are Cohen--Macaulay rings. As $\frakc$ is a complete
intersection, $R/\frakc$ is isomorphic to the canonical module of
$R/\frakc$ and hence
$\frakb/\frakc=\Hom_{R/\frakc}(R/\fraka,\omega_{R/\frakc})$ is a canonical
module for $R/\fraka$.

In the case at hand, $\frakc=I_W$, $\fraka=I_{W,u}$, $\frakb=I_{W,0}$. Since
$I_{W,u}$ is a complete intersection in its own right, $I_{W,0}$ is \emph{licci} (in the \emph{li}nkage \emph{c}lass of a \emph{c}omplete \emph{i}ntersection; see \cite{PeskineSzpiro-Liaison,HunekeUlrich-Annals}).
In particular, $I_{W,0}/I_W$ is isomorphic to the canonical module of
$\KK[u,x]/I_{W,u}$, which is (up to shift) $\KK[u,x]/I_{W,u}$ itself. It
follows that $I_{W,0}=I_W+(f)$ for a suitable $f\in\KK[u,x]$.
  
Writing $[q]=(q_1,\ldots,q_r)$ and $R=\KK[u,x]$,
  we now consider the free resolutions (Koszul
  complexes, both of length $r$) to $I_W$ and $I_{W,u}$, and the morphism induced between them
  by the equation $[q]=A[Q_E]A^\trp [u]$:
\[
  \begin{tikzcd}
    \bigwedge^r R^r\arrow[r,"\wedge^r({[q]})"]\arrow[d,dashed,"\wedge^r({A[Q_E]A^\trp})"] & \cdots\arrow[r,"\wedge^3({[q]})"]&
    \bigwedge^2 R^r\arrow[r,"\wedge^2({[q]})"]\arrow[d,dashed,"\wedge^2({A[Q_E]A^\trp})"] &
    \bigwedge^1 R^r\arrow[r,"{[q]}"]\arrow[d,"{A[Q_E]A^\trp}"] &
    \bigwedge^0 R^r=R\arrow[d,"1" ]\\
    \bigwedge^r R^r\arrow[r,"\wedge^r({[u]})"]& \cdots\arrow[r,"\wedge^3({[u]})"]&
    \bigwedge^2 R^r\arrow[r,"\wedge^2({[u]})"]&
    \bigwedge^1 R^r\arrow[r,"{[u]}"]&
    \bigwedge^0 R^r=R
\end{tikzcd}
\]
Since the rows of this diagram resolve $R/I_W$ and $R/I_{W,u}$ respectively,
the total complex is the mapping cone to the projection $R/I_W\to R/I_{W,u}$
and hence naturally identifies with $I_{W,u}:I_W$.  Since $R/I_{W,u}$ is
Cohen--Macaulay, so is $R/I_{W,0}$ and $I_{W,u}:I_W$ is its canonical module.
Thus, the dual of the diagram above yields a free resolution of
$R/I_{W,0}$. It is not minimal, but by pruning the vertical isomorphism in
the rightmost column we obtain a bigraded minimal resolution as the total complex of the following double complex
(with $\deg(u_i)=(0,1)$ and $\deg(x_j)=(1,0)$):
\begin{equation}\label{eq:minres}
  \begin{tikzcd}
    (\bigwedge^r R^r)(0,0)&
    \arrow[l,"\wedge^r({[q]})"] \cdots&
    \arrow[l,"\wedge^3({[q]})"](\bigwedge^2 R^r)(2-r,2-r)&
    \arrow[l,"\wedge^2({[q]})"](\bigwedge^1 R^r)(1-r,1-r)
    \\
    (\bigwedge^r R^r)(-r,0)\arrow[u,"\wedge^r(A{[Q_E]}A^\trp)" right]&
    \arrow[l,"\wedge^r({[u]})"]\cdots&
    \arrow[l,"\wedge^3({[u]})"](\bigwedge^2 R^r)(-r,2-r)
    \arrow[u,"\wedge^2(A{[Q_E]}A^\trp)"]&
    \arrow[l,"\wedge^2({[u]})"](\bigwedge^1 R^r)(-r,1-r) \arrow[u,"A{[Q_E]}A^\trp"]
\end{tikzcd}
\end{equation}

Thus, the free module $F_i$ of the $\ZZ$-graded resolution
$F_\bullet\onto R/I_{W,0}$, with
    $1\le i\le r-1$, is given by $F_i=(R(-2i))^{\binom{r}{i}}\oplus
    (R(-r-i+1))^{\binom{r}{i-1}}$, while $F_r=(R(1-2r))^{\binom{r}{r-1}}$. Since the largest shift in any $F_i$ is $2r-1$, and since the
    canonical module of $R$ is generated in degree $-(r+n)$,  
    the $a$-invariant of $R/I_{W,0}$ is $(2r-1)-(r+n)=r-1-n<0$.

We have proved most of the following result.


 \begin{thm}\label{thm-hatLambda-eqns}
 The coordinate ring $\KK[u,x]/I_{W,0}$ of $\hat\Lambda_{W,0}$  is a normal Cohen--Macaulay domain of type $r$ and of projective
 dimension $r$ over $\KK[u,x]$. It has Castelnuovo--Mumford regularity
 $r-1$ and the defining ideal is 
 $I_{W,0}=I_W+\ideal{\psi_W}$.
 \end{thm}

 \begin{proof}
 The last statement follows from $\bigwedge^r(A[Q_E]A^\trp)=[\psi_W]$ and is in fact a special case of \cite[Satz 2]{Wiebe}.

 Normality of $\hat\Lambda_{W,0}$ follows along the same lines as that  of $\Lambda_W$. Indeed, the proof of Corollary \ref{cor:normality} implies that  $\hat\Lambda_{W,0}$ is normal outside the vanishing locus of $u$. On the other hand, the Jacobian of the height $r$ ideal $(A[Q_E]A^\trp[u],\psi_W)$ in a point where $u$ vanishes is 
 \[
 \begin{pmatrix}
 0&\nabla\psi_W\\
 A[Q_E]A^\trp&0
 \end{pmatrix}.
 \]
At a  point of $\hat\Lambda_{W,0}$ where the first coordinate gives a  generic point on $\hat X_W$, the lower left block has rank $r-1$ and the upper right is nonzero. The claims about regularity and projective dimension are obtained from the minimal resolution \eqref{eq:minres}.
\end{proof}


\begin{rmk}
For disconnected $\matM_W$, the factorization of $\psi_W$ induces a decomposition of $\Lambda_W$ and $\hat\Lambda_{W,0}$. When $r>1$,
    $\hat\Lambda_{W,0}$ is not Gorenstein, since its coordinate ring has type $r$.
\end{rmk}

\subsubsection*{$F$-purity}

Let $R$ be a commutative ring of
finite characteristic $p>0$. The Frobenius endomorphism and its iterates
\begin{equation*}
\begin{tikzcd}[row sep=-5pt]
R\mathllap{F^e\colon\phantom{R}}\ar[r] & \mathrlap{\phantom{F^e_*(R)}}F^e_*(R)\\
x\ar[r,mapsto] & x^{p^e}
\end{tikzcd}
\end{equation*}
can be used to
classify singularities, where $F^e_*(R)$ is $R$ as an Abelian group, equipped with an $R$-module structure $r(r')=r^{p^e}r'$. (One traditionally uses $F^e_*(R)$ to indicate the
target of the $e$-th Frobenius on $R$ in order to easily differentiate
it from other instances of copies of $R$). 

For an ideal $J$ of 
$R$, denote by $J^{[p^e]}\subseteq J^{(p^e)}$ 
its $e$-th Frobenius power, the ideal generated by the $p^e$-th powers
of all elements of $J$. 


\begin{dfn}
  The ring $R$ is \emph{$F$-pure}, if the Frobenius morphism
  $F^e\colon R\to F^e_*(R)$ is pure, which is to say that its tensor
  product with every $R$-module yields an
  injective morphism.
\footnote{The choice of $e$ is mostly immaterial: if $F^e$ is pure for one
  $e>0$ then it is so for all $e>0$. Note in particular that
  $F$-purity requires that $F$ is injective itself, hence the ring must
  be reduced.}
\end{dfn}


The motivation for this definition comes from a theorem of Kunz that
states that a ring $R$ of prime characteristic is regular precisely when
$F^e$ is flat for some (equivalently all) $e>0$. In fact, for a local
ring, flatness implies that $F^e_*(R)$ is a free module over the image
ring $F^e(R)$. In that case, $F(R)$ is even a summand of $F^1_*(R)$
and thus $F$-pure rings can be viewed as \enquote{somewhat similar} to regular rings.

A related condition is \emph{$F$-injectivity}, determined by the injectivity
of the natural Frobenius action on all the local cohomology modules
$H^i_\frakm(R_\frakm)$ where $\frakm$ runs through the maximal ideals
of $R$. (For this action, view local cohomology as defined via a \v
Cech complex and take powers of fractions representing classes). In
general, $F$-purity implies $F$-injectivity and the converse holds for
Gorenstein rings.


\begin{rmk}
Suppose $R$ is a finitely generated algebra over a field $\KK$ of
characteristic zero. Lifting $R$ to a finitely generated $\ZZ$-algebra
$\tilde R$, one can discuss the properties of its reduction to various
primes $p$.

If $\KK=\CC$ and $X$ is a reduced finite type scheme over $\KK$, a somewhat
non-trivial condition (see \cite[Definitions~5.1,
  6.9]{Schwede-duBoisInjective}) determines whether $X$ is \emph{du
  Bois}. While the definition is complicated, it has some appealing
consequences. For example, on a proper reduced du Bois scheme $X$ the natural
map $H^i(X,\CC)\to H^i(X,\sO_X)$ is surjective. This is a geometric condition, generalizing the classical Hodge-theoretic facts that for smooth proper complex varieties, and for compact K\"ahler manifolds, this map is surjective. A very nice introduction to the relationship between du Bois and rational singularities is in Chapter 12 of \cite{Kollar-Shafarevich}.

It has been shown by Schwede in the same paper \cite{Schwede-duBoisInjective}, that $\Spec(R)$ has
du Bois singularities precisely when infinitely many
mod-$p$-reductions of $\tilde R$ are $F$-injective. 
\end{rmk}


An obvious complication for testing whether a given ring is $F$-pure
or not is that one ostensibly needs to check injectivity of infinitely
many maps. The following result reduces the problem to a finite computation.


\begin{lem}[Fedder's Criterion, \cite{Fedder}]
  If $(R,\frakm)$ is a regular local ring of characteristic $p>0$ and
  $I\subseteq R$ an ideal then $R/I$ is $F$-pure precisely when
  $I^{[p]}:I$ is not contained in $\frakm^{[p]}$.
\noproof
\end{lem}


There is a version of Fedder's Criterion (of the expected form) for
homogeneous ideals in standard graded polynomial rings,
\cite[Rmk.~2.9]{MaPolstra}.

In general, the ideal quotient $I^{[p]}:I$ is fairly mysterious, even when $I$ is nice. The
obvious exception arises when $I=(f_1,\ldots,f_k)$ is a complete
intersection in a UFD, in which case $I^{[p]}:I$ is generated by $\prod
f_i^{p-1}$. 


\begin{thm}\label{thm-cone-Fpure}
  Both $\hat\Lambda_W$ and the affine cone  $\hat\Lambda_{W,0}=\Spec(\KK[u,x]/(I_W+\ideal{\psi_W})$ over
  $\Lambda_W$ are $F$-pure over every field 
  of positive characteristic, and du Bois over every field of characteristic zero.
\end{thm}

\begin{proof}
  Recall that we may preprocess in such a way that the front $r$
  columns of $A$ form an identity matrix, and so the $i$-th component
  of $A[Q_E]A^\trp[u]$ has lead term $x_iu_i$ under the monomial order that
  refines the lex order in $x$ by lex in $u$.

  Let $Q$ be the product of all $q_i$, and write
  $\frakm\coloneqq(x_1,\ldots,x_n,u_1,\ldots,u_r)$.  Clearly, $Q^{p-1}\in
  (I_W^{[p]}:I_W)$ and $Q^{p-1}\not\in \frakm^{[p]}$, since its lead
  monomial $\prod_{i=1}^rx_i^{p-1}u_i^{p-1}$ is not in the monomial ideal
  $\frakm^{[p]}$. By Fedder's Criterion, $\KK[u,x]/I_W$ is $F$-pure.

  Now, $I_{W,0}=I_W+(\psi_W)$ arises as $I_W:I_{W,u}$, and that is enough to
  inherit $F$-purity from $\KK[u,x]/I_W$ to $\KK[u,x]/I_{W,0}$. Indeed,
  by \cite[Lem.~3.1]{PandeyTarasova}, one infers that $(I_W^{[p]}:I_W)$ is
  contained in $(I_{W,0}^{[p]}:I_{W,0})$ and thus $Q^{p-1}$ witnesses
  $F$-purity of $\KK[u,x]/I_{W,0}$ as well.
\end{proof}

\subsubsection{$F$-rationality and $F$-regularity}

\begin{dfn}
A ring is \emph{$F$-finite} if it is a finitely generated module over its Frobenius image. Fields are $F$-finite precisely when the extension degree  $[\KK^{1/p}:\KK]$ is finite. Finitely generated $\KK$-algebras over $F$-finite fields, as well as their localizations are $F$-finite, and so are complete local rings with with $F$-finite residue field. 

The $F$-finite ring $R$ is \emph{strongly $F$-regular} if for every $c\in R$ not in a minimal prime of $R$, there exists a natural number $e>0$ such that the morphism $R\to F^e_*(R)$ that is the composition of the Frobenius followed by multiplication by $F^e_*(c)\in F^e_*(R)$, splits as map of $R$-modules.
\end{dfn}


Strongly $F$-regular rings are Cohen--Macaulay domains, and the property is a local property. 
A related property is the following.


\begin{dfn}
Let $R$ be a local ring $R$ of dimension $d$ with maximal ideal $\frakm$, and consider the two conditions 
\begin{asparaenum}
\item for any $c$ in $R$ not in a  minimal prime of $R$ the map from $H^i_\frakm(R)$ to itself given by multiplication by $c$ following the $e$-fold iteration of the Frobenius, is injective, for any $i$;\label{it:F1}
\item $R$ is Cohen--Macaulay.\label{it:F2}
\end{asparaenum}
Then $R$ is \emph{$F$-injective}, if \eqref{it:F1} holds, and $F$-rational if \eqref{it:F1} and \eqref{it:F2} hold.
\end{dfn}


All four $F$-properties localize (see \cite{MaPolstra}). 
One notes that $F$-rationality relates to $F$-injectivity in a similar way strong $F$-regularity relates to $F$-purity. See the diagram at the start of Section 8 in \cite{MaPolstra} for relative strength of all these properties; strong $F$-regularity implies all others, and $F$-purity implies $F$-injectivity. 

Results of K.~Smith, N.~Hara and K.-I.~Watanabe \cite{Smith-AJM,Hara-AJM,HaraWatanabe} show intricate relations between strong $F$-regularity and $F$-rationality on one side and rational and log-terminal singularities on the other.

Below we identify $\hat\Lambda_{W,0}$ with its image under $\ell$, along the lines of Definition \ref{def:Lambda}.


\begin{lem}\label{lem-affine-duality}
For algebraically closed $\KK$, let $W\subseteq V$, $W^\perp\subseteq V^\star$ be dual configurations on $|E|=n$ elements, of respective ranks $r$ and $n-r$. Then $\hat\Lambda_{W,0}\cap (V\times \hat\TT_{E^\star})
\subseteq V\times  V^\star$ and $\hat\Lambda_{W^\perp,0}\cap (V^\star\times \hat\TT_E)
\subseteq V^\star\times V$ are isomorphic.
\end{lem}

\begin{proof}
Since the varieties are integral, it suffices to describe the isomorphism on closed points.

Let $A,A^\perp$ be full-rank matrices with row spans $W,W^\perp$ respectively, and let $D_{[\beta]}$ be a diagonal matrix as before. Then $A$ and $A^\perp$ minimally span each other's kernel, and $(v,\beta)\in \hat\Lambda_{W,0}$ if and only if a) $AD_{[\beta]} A^\trp$ is rank-deficient, b) $v$ is in the row span of $A$, and c) $AD_{[\beta]} [v]=0$. 

Assume that $\beta_i\neq 0$ for all $i\in E$ and define a morphism 
\begin{equation}\label{eqn-duality-map}
\begin{tikzcd}[row sep=-5pt]
\mathllap{\DD\colon}\hat\Lambda_{W,0}\cap (V\times \hat\TT_{E^\star})\ar[r] &
\hat\Lambda_{W^\perp,0}\cap (V^\star\times \hat\TT_E),\\
(v,\beta)\ar[r,mapsto] & (Q_E(\beta,v),1/\beta),\notag
\end{tikzcd}
\end{equation}
where $1/\beta=(1/\beta_1,\ldots,1/\beta_n)$, and $Q_E(-,-)$ is the Hadamard product from \eqref{eqn-hadamard}. Now, a) $A^\perp D_{[1/\beta]}{(A^\perp})^\trp$ is rank-deficient if and only if $AD_{[\beta]}  A^\trp$ is since that is where the determinants $\psi_W(\beta)=\psi_{W^\perp}(1/\beta)\cdot \prod_{i\in E}\beta_i$ vanish. Moreover, if $AD_{[\beta]} [v]=0$ then $A [Q_E(\beta,v)]=AD_{[\beta]} [v]=0$ shows that b) $Q_E(\beta,v)$ is in the row span of $A^\perp$. Finally, $v$ being in the row span of $A$ forces  c)
$A^\perp D_{[1/\beta]}[Q_E(\beta,v)]=A^\perp [v]=0$. Hence the image of $\DD$ is indeed within $\hat\Lambda_{W^\perp,0}\cap (V^\star\times \hat\TT_E)$.
The inverse morphism is given by $(Q_E(\beta^\perp,v^\perp),{\beta^\perp}^{-1})\mapsfrom(v^\perp,\beta^\perp)$.
\end{proof}


The following result of Schwede and Singh will be useful.


\begin{lem}[{\cite[Cor.~A4]{HMS-SS}}]\label{lem-HMS-SS}
An $F$-finite local ring $R$ is $F$-rational provided that  there exists a regular element $f\in R$ 
such that $R/\ideal{f}$  is $F$-injective and $R[1/f]$ is $F$-rational.
\noproof
\end{lem}


\begin{thm}\label{thm-cone-Frat}
Let $W$ be a configuration for the matroid $\M$ over an $F$-finite field of positive characteristic. If $\matM$ is connected, then $\hat \Lambda_{W,0}$ is $F$-rational, and in particular normal.
\end{thm}

\begin{proof}
It is enough to show this over algebraically closed fields by a result of V\'elez \cite[p.~440]{Velez1995} (see also \cite{DattaMurayama}). 
The normality claim follows from $F$-rationality in general by \cite[Thm.~4.2]{HH-TAMS94}. 

By Theorem \ref{thm-cone-Fpure}, $\hat\Lambda_{W,0}$ is $F$-pure for all $W$, connected or otherwise, and so  every local ring of $\hat\Lambda_{W,0}$ is also $F$-injective for all $W$. 
The blueprint of the proof of Theorem \ref{thm-cone-Frat} is the same as the one for \cite[Thm.~1.2]{BW2024}, by induction on the number $n=\abs{E}$. The main tool is Tutte's theorem (see \cite[Thm.\ 4.3.1]{Oxl11}) assuring that for a connected matroid $\matM$ and any edge $i\in E$, either $\matM\backslash i$ or $\matM/i$ is connected. 

The inductive base case is when the rank of $\matM_W$ is 1. Then it is immediate that the defining ideal of $\hat\Lambda_{W,0}$ is generated by $\psi_W$. But as $\matM$ is then $\matU_{1,n}$, $\psi_W$ is (linear, hence) smooth.

Now suppose that $M\backslash i$ is connected for some $i\in E$. Then $i$ is not a coloop of $\matM$ and hence $\rank(\matM)=\rank(\matM\backslash i)$. Thus, $[Q_W]$ and $[Q_{W\backslash i}]$ agree modulo $x_i$, and likewise the ideal $I_{W,0}$ describing $\hat\Lambda_{W,0}$ inside $V\times V^\star$ agrees modulo $x_i$ with $I_{W\backslash i,0}$. In other words, $\hat\Lambda_{W\backslash i,0}$ is the hyperplane section $\hat\Lambda_{W,0}\cap\Var(x_i)$ inside $V\times V^\star$. In the local ring at the origin, $F$-rationality of $\hat\Lambda_{W,0}$ then follows from that of $\hat\Lambda_{W\backslash i,0}$ at the origin by \cite[Thm.~4.2]{HH-TAMS94}.
This implies that $\hat\Lambda_{W,0}$ is $F$-rational in a neighborhood of the origin since the $F$-rational locus is open in our rings by a theorem of V\'elez, \cite{Velez1995}.  But since $\hat\Lambda_{W,0}$ is defined by a standard graded ideal, $F$-rationality near the origin implies global $F$-rationality.

We have reduced to showing that $\hat\Lambda_{W,0}$ is $F$-rational if $\matM$ and $\matM/i$ are connected. Since matroid duality exchanges restriction and deletion and also preserves connectedness, $\matM^\perp$ has the connected deletion $\matM^\perp\backslash i=(\matM/i)^\perp$. We may hence assume by induction on $|E|$ that the corresponding $\hat\Lambda_{W^\perp,0}$ is $F$-rational.  By Lemma \ref{lem-affine-duality}, the local rings of $\hat\Lambda_{W,0}$ in which  $\prod_{i\in E}x_i$ is a unit are thus all $F$-rational. 

Let $\hat\Lambda^\Frat_{W,0}$ be the $F$-rational locus of $\hat\Lambda_{W,0}$. 
Let $R=\KK[V\times V^\star]$, and 
for a prime ideal $\frakp$ containing the defining ideal $I_{W,0}$ of $\hat\Lambda_{W,0}$, set
\[
\nu(\frakp)\coloneqq\#\{i\in E\mid x_i\in \frakp\}.
\]
Then define, for $t\in\{0,\ldots,n\}$, the sets 
\[
P_t\coloneqq\{\frakp\in\Spec(R/I_{W,0})\mid \nu(\frakp)\le t\}.
\]
The previous paragraph shows for $t=0$ the implication $[\frakp\in P_t]\Rightarrow[\frakp\in\hat\Lambda^\Frat_{W,0}]$. We show now by induction corresponding statements for all $t$. For, suppose that $\frakp\in P_t\setminus P_{t-1}$ and choose $i\in E$ with $x_i\in\frakp$. Then the prime ideals $\frakq$ in $R$ that correspond to prime ideals in $(R/I_{W,0})[1/x_i]$ have $\nu(\frakq)<t$, and  it follows from induction that $(R/I_{W,0})_\frakq$ is $F$-rational. Since $F$-rationality is a local property,  the ring $(R/I_{W,0})_\frakp[1/x_i]$ is $F$-rational. On the other hand, reduction of $(R/I_{W,0})_\frakp$ by $x_i$ is the localization at $\frakp$ of $R/(I_{W,0}+Rx_i)$. But the latter is the local ring of $\hat\Lambda_{W\backslash i,0}$ at $\frakp (R/(x_i))$, which is $F$-pure and hence $F$-injective. By Lemma \ref{lem-HMS-SS}, $\hat\Lambda_{W,0}$ is $F$-rational at $\frakp$. This concludes the inductive step for the primes in $P_t$, and hence the proof of the theorem.
\end{proof}


\begin{rmk}
Suppose $Y$ is a variety of finite type over a field of  characteristic zero, and $\pi\colon \tilde Y\to Y$ is a resolution of singularities. Then $Y$ has \emph{rational singularities} if the canonical map to the derived direct image $\sO_Y\to R\pi_*(\sO_{\tilde Y})$ is a quasi-isomorphism. This entails normality of $Y$ as well as the vanishing of the  higher derived direct images. A rational singularity is du Bois and Cohen--Macaulay. We refer the reader to \cite{Kempf-survey} for some other aspects of rational singularities.
\end{rmk}


The following result is a relative of Theorem \ref{thm-cone-Fpure}. 
\begin{cor}\label{cor:ratsings}
In characteristic zero, for connected $\matM_W$, 
the affine cone $\hat\Lambda_{W,0}$ has rational singularities and is, in particular, normal.
\end{cor}

\begin{proof}
Given a configuration $W$ over a field $\KK$ of characteristic zero, lift it to an integer variety.  Then for infinitely many prime numbers $p$ the reduction modulo $p$ will be a configuration for $\matM$, and by Theorem  \ref{thm-cone-Frat} the corresponding $\hat \Lambda_{W,0}$ over $\ZZ/p\ZZ$ is an $F$-rational reduction of the model of $\hat \Lambda_{W,0}$ over $\KK$. This makes $\hat\Lambda_{W,0}$ over $\KK$ a variety of \emph{$F$-rational type}. 
K.~Smith and N.~Hara proved that $F$-rational type is equivalent to  rational singularities (see  \cite{Smith-AJM,TakagiWatanabe,HaraAJM}). 
\end{proof}


\begin{cor}\label{cor:strong_Fregular}
If $\matM$ is connected, then for any of its configurations $W$ over an $F$-finite field $\KK$, the image of $\hat\Lambda_W\setminus\Var(u)$ in $\PP W\times V^\star\subseteq \PP V\times V^\star$ is strongly $F$-regular.
\end{cor}

\begin{proof}
Strong $F$-regularity descends along field extensions, \cite[Lem.~3.17]{Hashimoto}. We may hence assume that $\KK=\bar\KK$. 

Away from $\Var(u)$, $I_{W,0}$ is the complete intersection of the entries $q_1,\dots,q_r$ of $[Q_W][u]$. But in Gorenstein rings, $F$-rationality is equivalent to strong $F$-regularity (see \cite[Section~4]{MaPolstra}).
\end{proof}

\begin{rmk}
When $\KK = \CC$, Corollary \ref{cor:strong_Fregular} implies that the variety therein, as well as $\Lambda_W \subseteq \PP W \times \PP V^\star$, have rational singularities. A concurrent proof appears in \cite[Cor.~4.1]{BM}, as a byproduct of a study of the singularities and contact loci of a larger genre of incidence varieties. Such methods cannot recover any of our $F$-singularity results, since contact loci are less understood in characteristic $p$.
\end{rmk}

\subsection{Nash blow-up and projective duality}\label{subsec:X_dual}

In this subsection, we assume that $\KK$ is algebraically closed of characteristic zero.
Bloch, Esnault and Kreimer~\cite[\S4]{BEK06} give an interpretation of $X_W$ as a dual variety under a very restrictive (and unstated) assumption that a certain map is an embedding. Details can be found in \cite[Prop.~2.2.2, Thm.~4.4.1]{PattersonThesis}. 
We interpret $X_W$ as a dual variety without any additional hypothesis.
Our investigation reveals the normalization of the Nash blow-up of $X_W$ to be $\Lambda_W$. 


First, we recall the general construction of incidence and dual varieties and Nash blow-ups.
Let $Z \subseteq \PP V$ be a projective variety and consider the projection $p_2\colon\PP V \times \PP V^\star \to \PP V^\star$.
The \emph{projectivized conormal bundle} of its smooth part $Z^\sm \subseteq \PP V$ is given by
\begin{equation*}
    \Xi_Z^\prime \coloneqq \set{(z, \beta) \in Z^{\sm} \times \PP V^\star \mid \beta\vert_{T_zZ}=0}.
\end{equation*}
Its Zariski closure in $\PP V\times\PP V^\star$ is the \emph{incidence variety} of $Z$, 
\begin{equation*}
    \Xi_Z \coloneqq \overline{\Xi_Z^\prime}\subseteq\PP V\times\PP V^\star
\end{equation*}
It is equidimensional of dimension
\begin{equation}\label{eq:dimXi}
\dim(\Xi_Z)=\dim(\Xi'_Z)=n-2
\end{equation}
with the same number of irreducible components as $Z$.
The \emph{dual variety} of $Z$ is given by (see~\cite[Ch.~1, 1.A, 3.A.(1)]{GKZ08}),
\begin{equation}\label{eq:Xi_dual}
Z^\vee \coloneqq\overline{\set{ \beta\in\PP V^\star \mid \exists z \in Z^\sm\colon\beta\vert_{T_zZ}=0}}= p_2(\Xi_Z)\subseteq \PP V^\star.
\end{equation}
Identifying $V=V^{\star\star}$, the biduality theorem (see \cite[Ch.~1, Thm.~1.1, (3.1), (3.2)]{GKZ08}) states that, 
\begin{equation}\label{eq:Xi_bidual}
\Xi_{Z^\vee}=\Xi_{Z},\quad Z^{\vee\vee}=Z.
\end{equation}

A related object is the Nash blow-up of $Z$, defined whenever $Z$ is equidimensional. 
From now on, we restrict to the case where $Z$ is a hypersurface.
The \emph{Gauss map} sends each smooth point of $Z$ to its tangent space:
\begin{equation*}
    \eta_Z: Z^\sm \to \PP V^\star, \quad z \mapsto T_zZ,
\end{equation*}
where we use $T_zZ \subseteq \PP V^\star$ as shorthand for the unique $\beta \in \PP V^\star$ vanishing on $T_zZ$. The \emph{Nash blow-up} is the Zariski closure in $Z\times \PP V^\star$ of its graph $\Gamma_{\eta_Z}=\Xi_Z'$, and agrees with the incidence variety (in the hypersurface case):
\begin{equation*}
    \Xi_Z=\Nash(Z) \coloneqq \overline{\Gamma_{\eta_{Z}}} \subseteq Z \times \PP V^\star.
\end{equation*}

\smallskip

Returning now to the setting of this paper, write $\ol{w}$ (resp. $\ol{v}$) for the image of $w \in W$ inside $\PP W$ (resp. the image of $v \in V$ inside $\PP V$).
Consider the morphism defined by the Hadamard square\
\begin{equation}\label{eq:square}
\iota_W\colon \PP W\to\PP V,\quad {w}={\sum_{j \in E} \ell_j(w) x_j}\mapsto {\sum_{j\in E}(\ell_j(w))^2x_j}={Q_W(w,w)}.
\end{equation}
Note that $\iota_W$ is finite since it is quasifinite and projective.
In \cite{BEK06} this map is assumed to be a closed immersion, which is not true in general.


\begin{exa}
Consider the configuration $W\subseteq V$ spanned by the rows of the matrix
\[
A=\begin{pmatrix}
1 & 1 & 0 & 0\\
0 & 0 & 1 & 1\\
1 & 2 & 4 & 8\\
\end{pmatrix}.
\]
Then $W$ represents the connected uniform matroid $\U_{3,4}$ but $\iota_W$ is not injective.
\end{exa}


The differential of $\iota_W$ is closely related to $Q_E^\star$ as follows.


\begin{lem}\label{lem:iota}
Pick coordinates $z_1,\dots,z_r$ corresponding to a basis $w^1,\dots,w^r$ of $W$.
For $i=1,\dots,r$, $w\in W$ and $\beta\in V^\star$, we have 
\[
\beta(\frac{\partial Q_W(w,w)}{\partial z_i})
=2Q^\star_W(w,\beta)(w^i).
\]
In particular, $(\ol{w},\beta)\in\Lambda_W$ if $(\iota_W(\ol{w}),\beta)\in\Xi_{Z_W}'$.
\end{lem}

\begin{proof}
By assumption, $w=\sum_{i=1}^rz_i(w)w^i$ for all $w\in W$.
Hence, $w_j=\sum_{i=1}^rz_i(w)w^i_j$ and $\frac{\partial}{\partial z_i}w_j=w^i_j$ for all $j\in\set{1,\dots,r}$.
The chain rule thus gives
\[
\frac{\partial Q_W(w,w)}{\partial z_i}
= \frac{\partial}{\partial z_i}\sum_{j\in E}w_j^2x_j
= \sum_{j\in E}2w_jw_j^ix_j.
\]
Applying $\beta\in V^*$ yields the claimed equality.
The particular claim follows since the projective image of $\frac{\partial Q_W(w,w)}{\partial z_i}$ belongs to $T_{\iota_W(w)}Z_W$ if $\iota_W(w)\in Z_W^\sm$.
\end{proof}


Since $\PP W$ is complete, $\iota_W$ is a closed morphism.
Its image, denoted by
\[
Z_W\coloneqq \iota_W(\PP W)\subseteq\PP V,
\]
is an irreducible variety with irreducible incidence variety $\Xi_{Z_W}$ projecting to the dual variety $Z_W^\vee$.
We show that $X_W$ and $Z_W$ are mutually dual:


\begin{prp}\label{prop:incidence}
Let $W\subseteq V$ be a connected configuration. 
Then $(\iota_W \times \id)(\Lambda_W) = \Xi_{Z_W}$, $Z_W^\vee=X_W$ and $X_W^\vee = Z_W$.
In particular, $\Xi_{Z_W}=\Xi_{X_W}=\Nash(X_W)$.
\end{prp}

\begin{proof}
By construction, $\Xi'_{Z_W}\subseteq Z_W^\sm\times\PP V^\star$ and hence
\[
(\iota_W\times\id)^{-1}(\Xi'_{Z_W})\subseteq(\iota_W\times\id)^{-1}(Z_W^\sm\times\PP V^\star)=\iota_W^{-1}(Z_W^\sm)\times\PP V^\star
\]
is a closed subvariety.
It is contained in $\Lambda_W$ by Lemma~\ref{lem:iota}.
As a consequence
\begin{equation}\label{eq:XiZ}
(\iota_W\times\id)^{-1}(\Xi'_{Z_{W}})\subseteq\Lambda_W\cap(\iota_W^{-1}(Z_W^\sm)\times\PP V^\star)\subseteq\Lambda_W
\end{equation}
is a (nonempty) closed subvariety of an open subset of $\Lambda_W$.
By Corollary~\ref{cor:Lambda_facts} and \eqref{eq:dimXi} for $Z=Z_W$, $\Lambda_W$ and the dense open are irreducible of dimension 
\[
\dim\Lambda_W=n-2=\dim\Xi'_{Z_W}.
\]
This equals the dimension of the left hand side of \eqref{eq:XiZ}, due to finiteness of $\iota_W$.
We conclude that the first inclusion in \eqref{eq:XiZ} is in fact an equality.
Since $\iota_W\times\id$ is projective, and hence closed, it commutes with taking Zariski closure.
Using irreducibility of $\Xi_W$ and $\Lambda_W$, it follows that
\begin{align*}
\Xi_{Z_W}=\overline{\Xi'_{Z_W}}
&=\overline{(\iota_W\times\id)\bigl((\iota_W\times\id)^{-1}(\Xi'_{Z_W})\bigr)}\\
&=\overline{(\iota_W\times\id)\bigl(\Lambda_W\cap(\iota_W^{-1}(Z_W^\sm)\times\PP V^\star)\bigr)}\\
&=(\iota_W\times\id)\bigl(\overline{\Lambda_W\cap(\iota_W^{-1}(Z_W^\sm)\times\PP V^\star)}\bigr)=(\iota_W\times\id)(\Lambda_W).
\end{align*}
Using \eqref{eq:Xi_dual} and Proposition~\ref{prop:pr2_over_Xsmooth} we find that
\[
Z_W^\vee=
p_2(\Xi_{Z_W})=
(p_2\circ(\iota_W\times\id))(\Lambda_W)=
p_2(\Lambda_W)=
X_W.
\]
By the biduality theorem \eqref{eq:Xi_bidual}, then
\[
X_W^\vee=Z_W^{\vee\vee}=Z_W,\quad\Xi_{Z_W}=\Xi_{Z_W^\vee}=\Xi_{X_W}=\Nash(X_W).\qedhere
\]
\end{proof}


We interpret $\Lambda_W$ as the normalized Nash blow-up $\wt{\Nash(X_W)}$ of $X_W$, show that both have rational singularities and describe the projection map $\Lambda_W\to X_W$.

\begin{thm} \label{thm:normalize} 
Let $W\subseteq V$ be a connected configuration. 
Then there is the following commutative diagram of irreducible varieties:
\begin{equation}\label{eq:Nash_diag}
\begin{tikzcd}
&\wt{\Nash(X_W)}\ar[dr,"\nu"]\\
\Lambda_W\ar[ur,"\phi","\cong"']\arrow{rr}{\iota_W\times \id}\arrow{dr}[swap]{p_2} && \Nash(X_W)\arrow[swap]{dl}[swap]{p_2} \\
& X_W
\end{tikzcd}
\end{equation}
Both $\Lambda_W$ and $X_W$ have rational singularities and  $R p_{2, *}\mathscr{O}_{\Lambda_W} = \mathscr{O}_{X_W}$.
\end{thm}

\begin{proof}
By Proposition \ref{prop:incidence}, the morphism $\iota_W \times \id$ in \eqref{eq:Nash_diag} is defined and surjective.
The lower triangle trivially commutes.

By Corollary \ref{cor:normality}, $\Lambda_W$ is an irreducible normal variety.
Since $X_W$ is an irreducible variety, so are $\Nash(X_W)$ and its normalization $\wt{\Nash(X_W)}$.
By the universal property of normalization (see \cite[Prop.~12.44]{GortzWedhornAG1}), the morphism $\iota_W \times \id$ in \eqref{eq:Nash_diag} thus uniquely factors through the normalization morphism $\nu$ of $\Nash(X_W)$ as a morphism $\phi$.
Note that $\phi$ is quasi-finite, since $\iota_W \times \id$ is so.

By Proposition \ref{prop:pr2_over_Xsmooth} and construction, $p_2$ on both $\Lambda_W$ and $\Nash(X_W)$ is an isomorphism over $X_W^{\sm}$, whereas $\nu$ is an isomorphism over the normal locus $\Nash(X_W)^{\nor}\supseteq p_2^{-1}(X_W^{\sm})$.
It follows that $\phi$ is an isomorphism over $(p_2 \circ \nu)^{-1}(X_W^\sm)$ and hence birational.
By Zariski's main theorem (see \cite[\S5.6: (N3)$\implies$(N4)]{MumfordOdaAGII}), it follows that $\phi$ is an isomorphism. 

The claims on rational singularities (which is an open condition) follow from Corollary \ref{cor:ratsings} and \cite[Theorems 1.1, 1.2]{BW2024} (later improved in \cite[Theorem 1.1]{BMW2024}), respectively. 

In order to check that $R p_{2,*} \mathscr{O}_{\Lambda_W} = \mathscr{O}_{X_W}$, pick a resolution of singularities $\tau\colon T \to \Lambda_W$.
By Proposition~\ref{prop:pr2_over_Xsmooth}, then also $p_2 \circ \tau\colon T \to X_W$ is a resolution of singularities. 
Consider the Grothendieck spectral sequence with second page
\begin{equation*}
    R^i p_{2,*} (R^j \tau_* \mathscr{O}_T) \implies R^{i+j} (p_2 \circ \tau)_{*} \mathscr{O}_{T}.
\end{equation*}
Since both $\Lambda_W$ and $X_W$ have rational singularities, the spectral sequence simplifies to
\[
Rp_{2, *} \mathscr{O}_{\Lambda_W} = R(p_2 \circ \tau)_{*} \mathscr{O}_T=\mathscr{O}_{X_W}.\qedhere
\]
\end{proof}


\begin{exa} \label{exa:delA3:minExp}
For the graphic configuration from Example \ref{ex:delA3}, we show that, over $\CC$, the birational modification $p_2: \Lambda_W \to X_W$ made the singularities milder.

For local complete intersections $Z$ embedded in a smooth 
$Y$, the singularity invariant par excellence is the minimal exponent $\tilde\alpha(Z, Y)$. This positive rational number is computed locally using Bernstein--Sato polynomials. The philosophy is that the larger the value of $\tilde\alpha(Z, Y) - \codim_Y(Z)$, the milder the singularities. Note that while minimal exponents depend on the embedding, \cite[Prop~4.14]{VFiltrationMinExpLCI} implies that $\tilde\alpha(Z, Y) - \codim_Y(Z)$ does not.

Let us compute $\tilde\alpha(\Lambda_W, \PP W \times \PP V^\star)$. 
Recall that the biprojective variety $\hat\Lambda_W \subseteq  W \times V^\star$ is cut out by the bihomogeneous defining ideal
\begin{equation*}
    \big(\underbrace{(x_1 + x_4 + x_5)u_1 + x_4 u_2 + x_3 u_3}_{=:h_1}, \, \underbrace{x_4u_1 + (x_2 + x_4)u_2}_{=:h_2}, \, \underbrace{x_5u_1 + (x_3 + x_5)u_3}_{=:h_3} \big)
\end{equation*}
inside $\CC[u_1, u_2, u_3, x_1, x_2, x_3, x_4, x_5]$.
Moreover, $\Lambda_W$ has only two singular points:
$(\alpha_{124},\beta_0)$ and $(\alpha_{135},\beta_0)$.  Let $Z =
\Lambda_W \setminus V(u_2, x_1) \subseteq \big(\PP W \times \PP
V^\star) \setminus V(u_2, x_1) \simeq \Spec(\CC[u_1, u_3, x_2, x_3,
  x_4, x_5]$; let $Z^\prime = \Lambda_W \setminus V(u_3, x_1)
\subseteq \big(\PP W \times \PP V^\star) \setminus V(u_3, x_1) \simeq
\Spec \CC[u_1, u_2, x_2, x_3, x_4, x_5]$. We find that $Z$ is
isomorphic to the hypersurface $D_3 \coloneqq (h_3 = 0) \subseteq
\Spec \CC[u_1, u_3, x_3, x_5]$. There is a formula for the
Bernstein--Sato polynomial of a homogeneous polynomial with an
isolated singularity; in this case, the minimal exponent is the number
of variables divided by the degree. So $\tilde\alpha(D_3, \CC^4) = 2$
and $\tilde\alpha(D_3, \CC^4) - \codim_{\CC^4}(D_3) = \tilde\alpha(Z,
\CC^6) - \codim_{\CC^6}(Z)$, which forces $\tilde\alpha(Z, \CC^6) =
4$. On the other hand, $Z^\prime$ is isomorphic to the hypersurface
$D_2 \coloneqq (h_2 = 0) \subseteq \Spec \CC[u_1, u_2, x_2, x_4]$. As
before we see that $\tilde\alpha(D_2, \CC^4) = 2$ and
$\tilde\alpha(Z^\prime, \CC^6) = 4$. These two computations on charts
stitch together, yielding $\tilde\alpha(\Lambda_W, \PP W \times \PP
V^\star) = 4$.

Using Macaulay2 \cite{M2} to compute Bernstein--Sato polynomials on
the five canonical affine patches of $\PP V^\star$ we see that
$\tilde\alpha(X_G, \PP V^\star) = 3/2$. We conclude the singularities
of $\Lambda_W$ are indeed nicer than those of $X_G$:
\begin{align*}
    1 = \tilde\alpha(\Lambda_W, \PP W \times \PP V^\star) - \codim_{\PP W \times \PP V^\star}(\Lambda_W) 
    >\tilde\alpha(X_G, \PP V^\star) - \codim_{\PP V^\star}(X_G) = 1/2 .
\end{align*}

We can say even more about the improvement of singularities using the
language of higher du Bois (resp. rational) singularities. For the
following statements, we use \cite[Thm~0.4]{SaitoOnBFunction},
\cite[Thm~F]{MPHodgeLocalCohomology},
\cite[Thm~1.3]{VFiltrationMinExpLCI}, \cite[Thm~1.1]{MinExpKRat}. From
$\tilde\alpha(X_G, \PP V^\star) - \codim_{\PP V^\star}(X_G) = 1/2$ we
conclude that $X_G$ has rational singularities but not the stronger
property of $1$-du Bois singularities. On the other hand, since
$\tilde\alpha(\Lambda_W, \PP W \times \PP V^\star) - \codim_{\PP W
  \times \PP V^\star}(\Lambda_W) = 1$ , we conclude that $\Lambda_W$
has $1$-du Bois singularities and, in particular, rational
singularities.

We will revisit this example in \S\ref{sec:normal}.
\end{exa}

\subsection{Classes and cohomology}\label{subsec:Lambda_classes}

In this subsection, we assume that $\KK=\CC$.
Results from the previous subsection serve to describe various classes and cohomology rings attached to $\Lambda_W$ in explicit combinatorial terms.


The stratification in Proposition~\ref{prop:Lambda_strata} yields the class of $\Lambda_W$ in the ring of varieties, proving Theorem~\ref{thm:motive_intro} from the introduction.


\begin{cor}\label{cor:Lambda_Grothendieck}
For any configuration $W\subseteq V$, the class of $\Lambda_W$ in the Grothendieck ring of varieties over $\KK$ equals 
\begin{align*}
  [\Lambda_W] &= \sum_{
    \substack{F\in\L_{\M_W} \\ F\neq E}}
  [\PP (W/F)^\circ]\times [\PP^{\abs{E}-\rank(\M_W\backslash F)-1}]\\
  &= \sum_{
    \substack{F\in\L_{\M_W} \\ F\neq E}}
  \overline{\chi}_{\M_W/F}(\LL)\cdot(\LL^{\abs{E}-\rank(\M_W\backslash F)}-1)/(\LL-1), 
\end{align*}
where $\LL=[\AA^1]$ and $\overline{\chi}_\M$ denotes the reduced characteristic polynomial of $\M$.
\end{cor}
\begin{proof}
  Aluffi~\cite[Thm.\ 1.1]{Alu13} shows that $[W^\circ]=
  \chi_{\M_W}(\LL)$ for any hyperplane arrangement complement $W^\circ$.
  Since $W^\circ\cong \CC^\times \times\PP W^\circ$ and
  $\chi_\M(t)=(t-1)\cdot\overline{\chi_{\M}}(t)$, the formula follows from
    Proposition~\ref{prop:Lambda_strata}.
\end{proof}


In contrast, the bidegree of $\Lambda_W$ depends only on $n$ and $r$.


\begin{cor}\label{cor:Lambda_Chow}
Let $W\subseteq V$ be a configuration of rank $r>0$. 
Then the class of $\Lambda_W$ in the Chow ring of $\PP V\times \PP V^\star$ is
\[
[\Lambda_W] = [H]^{n-r}([H]+[H^\star])^r,
\]
where $H\subseteq V$ and $H^\star\subseteq V^\star$ denote hyperplanes.
\end{cor}

\begin{proof}
By Corollary~\ref{cor:Lambda_facts}, $\Lambda_W\subseteq\PP V\times \PP V^\star$ is defined by a regular sequence consisting of $n-r$ generators of bidegree $(1,0)$ and $r$ generators of bidegree $(1,1)$.  (See, e.g., \cite[Ex.~8.4.2]{Ful98}).
The claimed equality follows.
\end{proof}


Bloch showed that the mixed Hodge structure on the cohomology of $\Lambda_W$ is mixed Tate (see \cite[Prop.~4.1]{Blo20}), that is, $H^{2k}(\Lambda_W,\CC)$ is pure of weight $2k$ for all $k\geq0$, and odd-dimensional cohomology is zero.
In the round case, the cohomology ring is determined combinatorially by $n$ and $r$.
Additively, this is clear, since a projective bundle over projective space has the Betti numbers of a product.  As an algebra, we have:


\begin{prp}\label{prp:Lambda_cohomology}
  For any round configuration $W\subseteq V$ of rank $r$, the Chow ring and
  cohomology ring of $\Lambda_W$ are given by 
\[
A^\bullet(\Lambda_W)=H^{2\bullet}(\Lambda_W,\ZZ)\cong \ZZ[a,b]/\big\langle a^r,
b^{-r}(b-a)^n\big\rangle,
\]
where $\deg(a)=\deg(b)=1$, and (formally)
\[
b^{-r}(b-a)^n=\sum_{k=0}^r \binom{n}{k}(-a)^k b^{n-r-k}.
\]
\end{prp}

\begin{proof}
The cycle class map between $A^\bullet$ and $H^{2\bullet}$ is an isomorphism~\cite[Ex.\ 19.1.11(d)]{Ful98}, so let us compute the Chow ring.  By Propositions~\ref{prop:Lambda_ProjSym} and \ref{prop:Lambda_strata}, $\Lambda_W\cong\Proj\Sym\sF$ is a $\PP^{n-r-1}$-bundle over $\PP W$.
Using \cite[Ex.\ 8.3.4]{Ful98},
\[
A^\bullet(\Lambda_W,\ZZ)\cong A^\bullet(\PP W)[b]/(b^{n-r}+c_1(\sF^\vee)b^{n-r-1}+\cdots+c_{n-r}(\sF^\vee)),
\]
where $b\coloneqq\sO_{\Lambda_W}(1)$ is the pullback of $[H^\star]-[H]$.
The presentation \eqref{eq:Lambda_sheaf} of $\sF$ yields the total Chern class
\[
\sum_{k=0}^nc_k(\sF^\vee)=c(\sF^\vee)=c(\sO_{\PP W}(-1)^n)/c(\sO_{\PP W}^r)=(1-a)^n=\sum_{k=0}^n{\binom{n}{k}}(-a)^k,
\]
where $a\coloneqq[H]\in A^\bullet(\PP W)$.
The claimed isomorphism follows.
\end{proof}
\begin{exa}[Example~\ref{ex:delA3}, continued]
  There are $12$ proper flats.  Summing in order of cardinality, 
  Corollary~\ref{cor:Lambda_Grothendieck} gives
  \begin{align*}
    [\Lambda_G] &= (\LL-2)^2[\PP^1] + (\LL-1)[\PP^1] + 4(\LL-2)[\PP^1] +
    4[\PP^1]  + 2[\PP^2]\\
  &=\LL^3+4\LL^2+2\LL+1.  
  \end{align*}
  With this we can also compute the class of the configuration hypersurface
  $[X_G]$ by writing $X_G = (C_{F_1}\triangle C_{F_2})\sqcup
  (C_{F_1}\cap C_{F_2})\sqcup X^\sm_G$.
  We recall that the fibre of $p_2$ is isomorphic to $\PP^1$ over the first
  two sets and a point over the last, so we obtain 
  \begin{align*}
    [X_G] &= [X^{\nsm}_G]+[X^\sm_G]\\
    &= (2\LL+1) + [\Lambda_G] - [p_2^{-1}(C_{F_1}\triangle C_{F_2})] -
    [p_2^{-1}(\beta_0)]\\
    &= (2\LL+1) + (\LL^3+4\LL^2+2\LL+1) - (\LL+1)(2\LL+1)\\
    &=\LL^3+2\LL^2+\LL+1.
  \end{align*}
  In view of the result of Belkale--Brosnan~\cite{BelkaleBrosnan},
  $[X_W]$ is not expected to be an integer class in general.  Exceptions
  occur: for example, in the round case, $[X_W]=[\Lambda_W]$.  For this
  non-round example, the singular locus of $X_G$ is a rational variety, by
  virtue of being one-dimensional, which makes the calculation above possible.
  It is disjoint from
  the torus $\TT_{E^\star}$, which is also not the case in general.
\end{exa}

\section{Tropical resolutions}\label{sec:tropical}

Our goal now is to construct an explicit resolution of singularities for the configuration hypersurface $X_W$ for general configurations $W$
using tropical techniques. 
In this section, we assume $\KK=\CC$, and that configurations $W$ are connected.
Then $X_W$ and $\Lambda_W$ are
irreducible varieties and can be considered as the closure of their
torus parts.  

\subsection{Restriction to tori}\label{sec:tori}

Let $\TT_{E,E^\star}\coloneqq\TT_E\times\TT_{E^\star}$ and $\hat\TT_{E,E^\star}\coloneqq\hat\TT_E\times\hat\TT_{E^\star}$ denote the respective tori in $\PP V\times \PP V^\star$ and $V\times V^\star$. 
Consider the torus parts of $\Lambda_W$ and $\hat\Lambda_W$,
\[
\Lambda_W^\circ\coloneqq\Lambda_W\cap\TT_{E,E^\star}\text{\quad and\quad}\hat\Lambda_W^\circ\coloneqq\hat\Lambda_W\cap\hat\TT_{E,E^\star}.
\]
The graph of the action of $\hat\TT_E$ on $\hat\TT_{E^\star}$ defines an automorphism
\[
\hat m\colon\hat\TT_{E,E^\star}\to\hat \TT_{E,E^\star},\quad(u,\beta)\mapsto(u,u\beta),
\]
equivariant with respect to the subtorus defining $\TT_{E,E^\star}$.
It induces an automorphism
\begin{equation}\label{eq:def_m}
m\colon\TT_{E,E^\star}\to \TT_{E,E^\star},\quad(u,\beta)\mapsto(u,u\beta).
\end{equation}


\begin{rmk}\label{rmk:m_coord_ring}
The automorphism $\hat m$ of $\hat\TT_{E,E^\star}$ corresponds to a bihomogeneous, monomial $\KK$-algebra automorphism of the coordinate ring:
\[
\hat m^\sharp\colon\KK[\hat\TT_{E,E^\star}]\to\KK[\hat\TT_{E,E^\star}],\quad y_i\mapsto y_i,\quad x_i\mapsto x_i/y_i.
\]
This follows from
\[
(u\beta)(x_i)=\beta(u^{-1}x_i)=\beta(y_i(u)^{-1}x_i)=y_i(u)^{-1}\beta(x_i)=(x_i/y_i)(u,\beta).
\]
Then the automorphism $m$ of $\TT_{E,E^\star}$ corresponds to a monomial $\KK$-algebra automorphism of the coordinate ring: 
\[
m^\sharp\colon\KK[\TT_{E,E^\star}]\to\KK[\TT_{E,E^\star}],\quad y_i/y_k\mapsto y_i/y_k,\quad x_i\mapsto (x_i/x_k)/(y_i/y_k).\qedhere
\]
\end{rmk}


\begin{ntn}
For any configuration $W\subseteq V$, we denote by
\[
\Lambda_W^\irr\coloneqq\overline{\Lambda_W^\circ}\subseteq\PP V\times \PP V^\star
\]
the closure of $\Lambda_W^\circ$ in $\PP V\times \PP V^\star$.
\end{ntn}


\begin{prp}\label{prop:Lambda_toric}
For any configuration $W\subseteq V$, $m$ and $\hat m$ induce isomorphisms 
\[
m\colon(\PP W)^\circ\times(\PP W^\perp)^\circ\to\Lambda_W^\circ,\quad
\hat m\colon W^\circ\times(W^\perp)^\circ\to\hat\Lambda_W^\circ.
\]
If $\M_W$ is a connected matroid 
then $\Lambda_W^\circ\subseteq\TT_{E,E^\star}$ is a smooth very affine variety with closure $\Lambda_W^\irr=\Lambda_W$.
\end{prp}

\begin{proof}
It suffices to prove the first claim for $\hat m$. 
By Definition~\ref{def:Lambda} and Lemma~\ref{lem:Q}, we have $(w,\beta)\in\hat\Lambda_W^\circ=\hat\Lambda_W\cap(\hat\TT_{E}\times\hat\TT_{E^\star})$ if and only if $w\in W\cap\hat\TT_E$ and $\beta\in\hat\TT_{E^\star}$ such that $0=Q_W^\star(w,\beta)=(w^{-1}\beta)\vert_W$.
Equivalently, $(w,w^{-1}\beta)\in(W\cap\hat\TT_E)\times(W^\perp\cap\hat\TT_{E^\star})$, and by definition,  $\hat m(w,w^{-1}\beta)=(w,\beta)$.

Since we assume $n>r>0$, because $\M_W$ is connected, 
it has an edge that is neither loop or coloop. 
Then $\PP W^\circ$ and $(\PP W^\perp)^\circ$ are nonempty and smooth.
The image $\Lambda_W^\circ$ of $m$ is thus dense open in $\Lambda_W$, and smooth due to the first claim.
By Corollary~\ref{cor:Lambda_facts}, $\Lambda_W$ is irreducible.
Then also $\Lambda_W^\circ$ is irreducible with closure $\ol{\Lambda_W^\circ}=\Lambda_W$.
\end{proof}


\begin{cor}\label{cor:Lambda_birat_X}
Let $W\subseteq V$ be a connected configuration. 
Then there is a commutative square of birational maps of irreducible varieties
\[
\begin{tikzcd}
\Lambda_W^\circ\ar{d}{p_2}\ar[hookrightarrow]{r} & \Lambda_W\ar{d}{p_2}\\
X_W^\circ\ar[hookrightarrow]{r} & X_W.
\end{tikzcd}
\]
\end{cor}

\begin{proof}
Commutativity holds trivially.
By Propositions~\ref{prop:pr2_over_Xsmooth} and \ref{prop:Lambda_toric}, the right projection and the upper inclusion are birational morphisms of irreducible varieties.
Then the same holds for the remaining morphisms.
\end{proof}


In \S\ref{sec:trop_comp} we will replace $\PP V\times \PP V^\star$ in the definition of $\Lambda_W^\irr$ by a finer $\TT_{E,E^\star}$-toric variety in order to obtain a smooth model of the configuration hypersurface $X_W$ for general configurations $W$.


\begin{prp}\label{prop:semigroup_algebra}
For any configuration $W\subseteq V$, the biprojective coordinate ring of $\Lambda_W^\irr$ is isomorphic to the $\KK$-subalgebra 
\[
L_W\coloneqq\KK[\ell_1,\ldots,\ell_n,\ell^\perp_1/\ell_1,\ldots,\ell^\perp_n/\ell_n]\subset\KK(W\times W^\perp).
\]
\end{prp}

\begin{proof}
Denote the respective coordinate rings of $V\times V^\star$, $\hat\TT_{E,E^\star}$ and $W^\circ\times(W^\perp)^\circ$ by
\begin{align*}
R&\coloneqq\KK[V\times V^\star]\cong\KK[y_1,\dots,y_n,x_1,\dots,x_n],\\ 
S&\coloneqq\KK[\hat\TT_{E,E^\star}]\cong\KK[y_1^{\pm1},\dots,y_n^{\pm1},x_1^{\pm1}.\dots,x_n^{\pm1}],\\
T&\coloneqq\KK[W^\circ\times(W^\perp)^\circ]\cong\KK[\ell_1^{\pm1},\dots,\ell_n^{\pm1},(\ell^\perp_1)^{\pm1}.\dots,(\ell^\perp_n)^{\pm1}].
\end{align*}
Composing the automorphism $\hat m^\sharp$ from Remark~\ref{rmk:m_coord_ring} with the inclusion $R\into S$ and the restriction to $W^\circ\times(W^\perp)^\circ$ yields a map of $\KK$-algebras 
\[
\phi\colon R\into S\overunderset{\hat m^\sharp}{\cong}{\longrightarrow}S\onto T,\quad y_i\mapsto\ell_i,\quad x_i\mapsto \ell_i^\perp/\ell_i,
\]
with image $L_W$.
Denote by $I\coloneqq \ker\phi$ its kernel.
The induced isomorphism $\ol\phi$ fits into a commutative square of $\KK$-algebras
\[
\begin{tikzcd}
R/I\ar[hook]{d}\ar{r}{\bar\phi}[swap]{\cong} & L_W\ar[hook]{d}\\
S/SI\ar[hook]{r} & T.
\end{tikzcd}
\]
By construction, the bottom map induces the map of coordinate rings associated with the isomorphism
\[
\hat m\colon W^\circ\times(W^\perp)^\circ\to\hat\Lambda_W^\circ
\]
in Proposition~\ref{prop:Lambda_toric}.
Since it is injective and maps onto $T$, it is isomorphism.
Then $SI$ is the ideal of $\hat\Lambda_W^\circ\subseteq\hat\TT_{E,E^\star}$ and pulls back to the ideal $R\cap SI=I$ of $\hat\Lambda_W^\circ\subseteq V\times V^\star$.
Since $\hat\Lambda_W^\circ$ is a biconical subset, $I$ agrees with the ideal of $\Lambda_W^\circ\subseteq\PP V\times\PP V^\star$, defining its closure $\Lambda_W^\irr$.
This turns $R/I$ into the biprojective coordinate ring of $\Lambda_W^\irr$, and $\bar\phi$ becomes the desired isomorphism.
\end{proof}

\subsection{Fans and tropicalization}\label{subsec:fans}

Tropical geometry gives a framework for studying toric closures of very affine varieties, such as $\Lambda_W^\irr$, by means of fans. 
For a very affine variety $Y\subseteq\TT_E$, the tropicalization $\trop(Y)$ (with respect to the trivial valuation on the field) is the underlying set of a rational, polyhedral fan, together with integer weights on the maximal cones.
The fan structure on the tropicalization is not unique, but by choosing a sufficently fine 
structure one may arrange it to be smooth.
We use Proposition~\ref{prop:Lambda_toric} to describe the tropicalization of $\Lambda_W^\circ$ in terms of \emph{Bergman fans}.
We begin by recalling the relevant definitions. 


Consider $e_i\coloneqq x_i$, $i\in E$, as the unit vectors of a lattice $\ZZ^E\subseteq V$.
For each subset $S\subseteq E$, we write an indicator vector $e_S\coloneqq \sum_{i\in S}e_i$.
We denote the co-character lattice of $\TT_E$ and its associated real vector space by
\[
N_{E,\ZZ}\coloneqq \Hom(\KK^\times,\TT_E)\cong\ZZ^E/\ZZ e_E\quad\text{and}\quad N_E\coloneqq N_{E,\ZZ}\otimes_\ZZ\RR\cong\RR^E/\RR e_E.
\]
Negation is an automorphism defined over $\ZZ$
\[
-\colon N_E\to N_E
\]
sending $e_S\mapsto e_{E\backslash S}$ for all $S\subseteq E$.
Define $e_i^\star$, $N_{E,\ZZ}^\star$ and $N_E^\star$ using $\TT_{E^\star}$ and $y_i$ accordingly.
We identify the coordinate tori $\hat\TT_E$ and $\hat\TT_{E^\star}$ by the requirement $x_iy_i=1$.
The resulting identification of co-character lattices then reads
\begin{equation}\label{eq:identify_lattices}
N_{E,\ZZ}^\star=N_{E,\ZZ},\quad e_i^\star=-e_i.
\end{equation}
The purpose of this identification will become clear in the following Lemma~\ref{lem:m_mu} and Proposition~\ref{prop:trop_Lambda}.
We write 
\[
N_{(E,E),\ZZ}\coloneqq N_{E,\ZZ}\oplus N_{E,\ZZ}\quad\text{and}\quad N_{E,E}\coloneqq N_{(E,E),\ZZ}\otimes_\ZZ\RR=N_E\oplus N_E
\]
for the the co-character lattices and associated real vector spaces of both $\TT_E\times\TT_E$ and $\TT_{E,E^\star}$.
The indicator vectors in the two summands of $N_{(E,E),\ZZ}$ will be denoted by $e_S$ and $f_S$, respectively.
We consider the maps
\[
\pi_1,\pi_2 \colon N_{E,E}\to N_{E}\quad\text{and}\quad\delta\colon N_{E,E}\to N_{E}
\]
defined over $\ZZ$, the
two projection maps $\pi_1$ and $\pi_2$ given by $\pi_1(x,y)=x$ and $\pi_2(x,y)=y$, respectively, and the addition map $\delta$ given by $\delta(x,y)=x+y$.


For a rational, polyhedral fan $\Sigma$, we denote by $\abs{\Sigma}$ its support and by $\PP(\Sigma)$ its toric variety.
The simplicial, unimodular fans are called smooth because they correspond to smooth toric varieties.


\begin{exa}\label{exa:fans}
We collect the examples of fans that we will relevant for us.
\begin{enumerate}[(a)]

\item\label{it:PV-fan} The fan $\Gamma_E$ is the smooth complete fan in $N_E$ whose cones are spanned by all proper subsets of the coordinate vectors $\set{e_i\mid i\in E}$.
The associated toric variety $\PP(\Gamma_E)\cong\PP V$ is the projective space of $V$ and $\PP(-\Gamma_E)\cong\PP V^\star$ its dual by \eqref{eq:identify_lattices}.

\item\label{exa:permutohedral_fan}  The \emph{permutohedral fan} $\Sigma_E$ is a smooth complete fan in $N_E$.
Its rays are spanned by indicator vectors $e_S$ for which $S\subseteq E$ is a nonempty, proper subset.
The cones
\[
\sigma_{\bS}\coloneqq\sum_{i=1}^k\RR_{\geq0}\cdot e_{S_i}\subseteq N_E
\]
are indexed by strict flags of $k$ distinct, nonempty proper subsets
\[
\bS=(\emptyset\subsetneq S_1\subsetneq S_2\subsetneq \cdots\subsetneq S_k\subsetneq E),
\]
where $k\geq 0$.
Since $-e_S=e_{E\backslash S}$, $\Sigma_E$ is stable under negation, that is, $-\Sigma_E=\Sigma_E$.

\item\label{exa:Bergman_fan} 
For a matroid $\M$ on $E$ without loops, the \emph{Bergman fan} $\Sigma_{\M}$ (see \cite{AK06,FS05}) is an induced and hence smooth subfan of $\Sigma_E$ of dimension $\rank\M-1$.
Its rays are spanned by indicator vectors $e_F$ for which $F$ is a nonempty, proper flat of $\M$.
The same holds for the subfan $-\Sigma_{\M}$ of $-\Sigma_E=\Sigma_E$.
In particular, $-\Sigma_{\M}\times\Sigma_{\M^\perp}$ is a (smooth) induced subfan of $\Sigma_E\times\Sigma_E$. 

\item \label{exa:bipermutohedral_fan}
The \emph{bipermutohedral fan} $\Sigma_{E,E}$ is a smooth, complete fan in $N_{E,E}$, which refines the product fan $\Sigma_E\times\Sigma_E$ and admits maps of fans (see \cite[Props.~2.2,~2.11]{ADH23})
\[
\pi_1,\pi_2\colon \Sigma_{E,E}\to\Sigma_{E}\quad\text{and}\quad\delta\colon \Sigma_{E,E}\to\Gamma_E.
\]
Its rays are generated by vectors $e_S+f_T$ indexed by \emph{bisubsets} $S\vert T$.
These are, by definition, pairs of nonempty sets $S,T\subseteq E$ for which $S\cup T=E$ and $S\cap T\neq E$.
Its cones
\[
\sigma_{\bS\vert\bT}\coloneqq\sum_{i=1}^k\RR_{\geq0}\cdot e_{S_i\vert T_i}\subseteq N_{E,E}
\]
are indexed by all \emph{biflags} $\bS\vert\bT$.
These are pairs of flags
\[
\bS=(S_1 \subseteq S_2 \subseteq \cdots \subseteq S_k)\quad\text{and}\quad
\bT=(T_1 \supseteq T_2 \supseteq \cdots \supseteq T_k),
\]
regarded as collections of bisubsets $\set{S_i\vert T_i\mid 1\leq i\leq k}$, for which $\bigcup_{i=1}^k(S_i\cap T_i)\neq E$. 
\end{enumerate}
\end{exa}


\begin{lem}\label{lem:m_mu}
The tropicalization of the automorphism \eqref{eq:def_m} is a lattice automorphism
\[
\mu\coloneqq\trop(m)=\id_{N_E}\oplus\,\delta\colon N_{E,E}\to N_{E,E}
\]
which satisfies $\pi_2\circ\mu=\delta$.
\end{lem}

\begin{proof}
According to Remark~\ref{rmk:m_coord_ring}, the automorphism $m^\sharp$ corresponding to \eqref{eq:def_m} on coordinate rings is a monomial map defined on exponent vectors by $(x,y)\mapsto(x,y-x)$.
The contragredient action from \eqref{eq:identify_lattices} negates the second coordinate, giving the map (see \cite[p.~87]{MS15})
\[
(x,y)=(x,-(-y))\mapsto(x,x-(-y))=(x,x+y).\qedhere
\]
\end{proof}


\begin{ntn}\label{not:Delta_M}
For any loopless matroid $\M$ on $E$, consider the smooth fan in $N_{E,E}$
\[
\Delta_\M\coloneqq-\mu((-\Sigma_{\M})\times\Sigma_{\M^\perp}).\qedhere
\]
\end{ntn}


\begin{prp}\label{prop:trop_Lambda}
Let $W$ be a configuration whose matroid $\M=\M_W$ has no loops or coloops.
Then $\trop(\Lambda_W^\circ)=\abs{\Delta_\M}$ with all weights equal to $1$.
\end{prp}

\begin{proof}
By Proposition~\ref{prop:Lambda_toric} and Lemma~\ref{lem:m_mu}, we have (see \cite[Cor.\ 3.2.13, (5.5.7)]{MS15})
\begin{align*}
  \trop(\Lambda_W^\circ)&=\trop(m)\big(
  \trop((\PP W)^\circ \times (\PP W^\perp)^\circ)\big)\\
&=\mu\big(\trop((\PP W)^\circ)\times\trop((\PP W^\perp)^\circ)\big).
\end{align*}
The set $\trop((\PP W)^\circ)\subseteq N_E$ is the support of the Bergman fan $\Sigma_{\M}$ (see \cite[Lem.~6.5]{FS05}, \cite[Thm.~1]{AK06}). 
Similarly, using \eqref{eq:identify_lattices}, we identify $\trop((\PP W^\perp)^\circ)\subseteq N_E^\star$ with the support of the fan $-\Sigma_{\M^\perp}$ in $N_E$.
The weights on Bergman fans all equal $1$ (see \cite[Prop.~5.2]{AHK18}). 
This property is preserved under products of fans (see \cite[Ex.\ 2.9.(iii)]{GKM09}) and under lattice automorphisms.
The claim follows.
\end{proof}

\subsection{Tropical compactifications}\label{sec:trop_comp}

We recall the following notions introduced by Tevelev~\cite{Tev07}.


\begin{dfn}
A compactification $\overline{Y}$ of a very affine variety $Y\subseteq\TT$ in a $\TT$-toric variety $\PP$ is called \emph{tropical}/\emph{sch\"on} if $\overline{Y}$ is proper and the torus multiplication map $\TT\times\overline{Y}\to\PP$ is surjective and flat/smooth.
If such a compactification exists, then also $Y\subseteq\TT$ is called \emph{tropical}/\emph{sch\"on}.
\end{dfn}


\begin{rmk}
Hacking~\cite[Lem.~2.7]{Hac08} observed that $\ol Y$ being schön is equivalent to smoothness of $\ol Y\cap O$ for all orbits $O$ in $\PP$. In particular, $Y=\ol Y\cap\TT$ must be smooth in order for $\ol Y$ to be sch\"on. 
\end{rmk}


Tevelev proved the following fundamental results.


\begin{thm}[{\cite[Thms.~1.2, 1.4, Prop.~2.5]{Tev07}}]
Any very affine variety $Y\subseteq\TT$ tropically compactifies in some $\TT$-toric variety $\PP$.
If it does so in $\PP=\PP(\Sigma)$ for some rational, polyhedral fan $\Sigma$, then $\abs\Sigma=\trop(Y)$, and it also does so in $\PP(\Sigma')$ for any refinement $\Sigma'$ of $\Sigma$.
In particular, there is always a tropical compactification in a smooth $\PP$.
If one tropical compactification of $Y$ is sch\"on, then all are.
\noproof
\end{thm}


Our starting point is the following observation.


\begin{prp}\label{prop:Lambda_schoen}
For any configuration $W\subseteq V$, the very affine variety $\Lambda_W^\circ\subseteq\TT_{E,E^\star}$ is sch\"on.
\end{prp}

\begin{proof}
Any hyperplane arrangement complement $(\PP W)^\circ\subseteq\TT_E$ is schön due to \cite[Thm.~1.5]{Tev07}.
By Proposition~\ref{prop:Lambda_toric}, $\Lambda_W^\circ\subseteq\TT_{E,E^\star}$ is isomorphic to a product of such under the torus automorphism in \eqref{eq:def_m}.
Sch\"on compactifications are preserved under products and torus automorphisms.
The claim follows.
\end{proof}


A key result due to Luxton and Qu makes it easy to find schön compactifications.

\begin{thm}[{\cite[Thm.~1.5]{LQ11}}]\label{thm:schoen_fan}
Let $Y\subseteq\TT$ be a schön very affine variety.
Then any rational, polyhedral fan $\Sigma$ with $\abs\Sigma=\trop(Y)$ gives rise to a sch\"on compactification $\ol Y$ in $\PP(\Sigma)$.\noproof
\end{thm}


For our purpose, the following well-known consequence is useful.

\begin{cor}\label{cor:schoen_smooth}
Let $Y\subseteq\TT$ be a schön very affine variety.
Then any smooth rational, polyhedral fan $\Sigma$ with $\abs\Sigma=\trop(Y)$ gives rise to a smooth sch\"on compactification $\ol Y$ in $\PP(\Sigma)$.
\end{cor}

\begin{proof}
By hypothesis and Theorem~\ref{thm:schoen_fan}, both $\PP\coloneqq\PP(\Sigma)\to\Spec\KK$ and $\TT\times\overline{Y}\to\PP$ are smooth, hence so is the composite $\TT\times\overline{Y}\to\Spec\KK$. 
Base change of the smooth map $\TT\to\Spec\KK$ along $\ol Y\to\Spec\KK$ shows that $\TT\times\overline{Y}\to\overline{Y}$ is surjective and smooth.
Then it follows from the diagram
\[
\begin{tikzcd}[column sep=small]
\TT\times\ol{Y}\ar[rr]\ar[dr] &&\ol{Y}\ar[dl] \\
& \Spec\KK
\end{tikzcd}
\]
that $\overline{Y}\to\Spec\KK$ is too (see \cite[\href{https://stacks.math.columbia.edu/tag/02K5}{Lemma 02K5}]{stacks}).
\end{proof}


In order to obtain a smooth model of the configuration hypersurface $X_W$ for general configurations $W$, we now replace $\PP V\times \PP V^\star$ in the definition of $\Lambda_W^\irr$ with a general $\TT_{E,E^\star}$-toric variety (see \S\ref{sec:tori}).


\begin{ntn}\label{ntn:toric_Lambda}
For any configuration $W\subseteq V$ and any rational, polyhedral fan $\Sigma$ in $N_{E,E}$, we denote by
\[
\Lambda_W(\Sigma)\coloneqq\Lambda_W(\PP(\Sigma))\coloneqq\overline{\Lambda_W^\circ}\subseteq\PP(\Sigma)
\]
the irreducible subvariety obtained as the closure of $\Lambda_W^\circ$ in the $\TT_{E,E^\star}$-toric variety $\PP(\Sigma)$ (see Proposition~\ref{prop:Lambda_toric}).
\end{ntn}


\begin{rmk}\label{rmk:Lambda_orig}
For any connected configuration $W\subseteq V$, we have 
\[
\Lambda_W=\Lambda_W(\PP V\times\PP V^\star)=\Lambda_W(\Gamma_E\times(-\Gamma_E)),
\]
due to Proposition~\ref{prop:Lambda_toric} (see Example~\ref{exa:fans}.\eqref{it:PV-fan}).
\end{rmk}


We are interested in fans $\Sigma$ for which $\Lambda_W(\Sigma)$ is not only
smooth, but also maps to $X_W$.  The correct notion is the following.


\begin{dfn}\label{def:trop_res}
We will say that a smooth fan $\Delta$ in $N_{E,E}$ defines a \emph{tropical resolution} for a matroid $\M$, or for a configuration $W\subseteq V$ with matroid $\M=\M_W$, if $\abs\Delta=\abs{\Delta_\M}$ and the projection $-\pi_2\colon N_{E,E}\to N_E$ defines a map of fans $\Delta\to\Gamma_E$.
We call it \emph{biprojective}, if also $\pi_1$ defines such a map of fans. 
Equivalently, this means that the identity of $N_{E,E}$ defines a map of fans
$\Delta\to\Gamma_E\times(-\Gamma_E)$.
\end{dfn}


We note that the smooth fan $\Delta_\M$ itself is not in general a tropical
resolution, because it fails the condition on $\pi_2$.
In \S\ref{sec:normal}, we rectify the defect by producing a suitable refinement.


\begin{thm}\label{thm:trop_res}
Let $W\subseteq V$ be a connected configuration. 
If $\Delta$ defines a tropical resolution for $W$, then $\Lambda_W(\Delta)$ is a smooth schön compactification of $\Lambda_W^\circ\subseteq\TT_{E,E^\star}$ and $\pi_2\colon N_{E,E}\to N_E$ induces a birational surjection $p\colon\Lambda_W(\Delta)\to X_W$.
\end{thm}

\begin{proof}
By Propositions~\ref{prop:trop_Lambda}, \ref{prop:Lambda_schoen} and Corollary~\ref{cor:schoen_smooth}, $\Lambda_W^\circ\subseteq\TT_{E,E^\star}$ is sch\"on, and the smooth $\TT_{E,E^\star}$-toric variety $\PP(\Delta)$ makes $\Lambda_W(\Delta)$ a smooth tropical compactification.

By hypothesis, $\pi_2$ defines a toric morphism $p\colon\PP(\Delta)\to\PP V^\star$ (see Example~\ref{exa:fans}.\eqref{it:PV-fan} and \cite[Thm.~3.4.11]{CLS11}), which fits into a commutative diagram:
\[
\begin{tikzcd}
\Lambda_W^\circ\ar{d}{}\ar[hook]{r} & \Lambda_W(\Delta)\ar[dashed]{d}\ar[hook]{r} & \PP(\Delta)\ar{d}{p}\\
X_W^\circ\ar[hook]{r} & X_W\ar[hook]{r} & \PP V^\star 
\end{tikzcd}
\]
Each row contains a very affine variety (left) and its closure (middle) in a toric variety (right) (see Proposition~\ref{prop:X_very_affine}).
The left vertical map is birational (see Corollary~\ref{cor:Lambda_birat_X}).
Since $p$ is continuous, it induces the dashed map, which is closed because
$\Lambda_W(\Delta)$ is proper.
The image of $\Lambda_W(\Delta)$ is then closed in $X_W$, and contains a dense open subset.
It follows that the dashed map is a birational surjection.
\end{proof}


The construction of $\Lambda_W(\Sigma)$ is functorial in the following sense.

\begin{prp}\label{prop:trop_res_natural}
Let $W\subseteq V$ be a connected configuration. 
Suppose that the identity of $N_{E,E}$ defines a map $\Sigma\to\Sigma'$ of rational, polyhedral fans such that $\pi_2$ defines a map of fans $\Sigma'\to-\Gamma_E$.
Then there is an induced commutative diagram:
\[
\begin{tikzcd}[column sep=small]
\Lambda_W(\Sigma)\ar{rr}{f}\ar["p",dr,swap] && \Lambda_W(\Sigma')\ar["p'",dl]\\
& X_W
\end{tikzcd}
\]
If $\abs{\Delta_{\M}}\subseteq\abs{\Sigma}$, then $f$ and $p$ are birational surjections.
\end{prp}

\begin{proof}
We argue as in the proof of Theorem~\ref{thm:trop_res}.
Since $f$ is the identity on $\Lambda_W^\circ$, the maps exist by continuity of the induced toric morphisms, and $p'\circ f=p$.

By a result of Tevelev (see \cite[Prop.~2.3]{Tev07}), the condition on supports $\trop(\Lambda_W^\circ)=\abs{\Delta_\M}\subseteq\abs{\Sigma}$ (see Proposition~\ref{prop:trop_Lambda}) makes $\Lambda_W(\Sigma)$ proper and thus $f$ and $p$ surjective.
\end{proof}

\begin{cor}\label{cor:trop_res_factor}
Let $W\subseteq V$ be a connected configuration. 
If $\Delta$ defines a biprojective tropical resolution for $W$, then the morphism $p\colon\Lambda_W(\Delta)\to X_W$ induced by $\pi_2$ factors through $p_2\colon\Lambda_W\to X_W$.
\end{cor}

\begin{proof}
Apply Proposition~\ref{prop:trop_res_natural} to $\Sigma=\Delta$ and $\Sigma'=\Gamma_E\times(-\Gamma_E)$ (see Remark~\ref{rmk:Lambda_orig}).
The required map of fans is given by definition of a biprojective tropical resolution.
\end{proof}

\subsection{Square conormal fans}\label{sec:normal}

Now we give a combinatorial recipe to construct a tropical resolution for
any matroid.
Our construction uses the bipermutohedral fan (see Example~\ref{exa:fans}.\eqref{exa:bipermutohedral_fan}).

  
\begin{dfn}\label{dfn:normal}
For a matroid $\M$ on $E$ without loops or coloops, we define the \emph{square conormal fan} of $\M$ as the (cone-wise) intersection of fans
\[
\Sigma_{-\M,\M^\perp}\coloneqq((-\Sigma_{\M})\times\Sigma_{\M^\perp})\cap\Sigma_{E,E}.
\]
\end{dfn}


\begin{rmk}
The definition of the square conormal fan closely resembles that of the \emph{conormal fan} $\Sigma_{\M,\M^\perp}=(\Sigma_{\M}\times\Sigma_{\M^\perp})\cap\Sigma_{E,E}$, which plays an important role in \cite{ADH23};
however, these two fans are not isomorphic.  They are tropicalizations of
incidence varieties coming from the Hadamard square immersion~\eqref{eq:square} 
in the former case, and from the logarithmic immersion of $\PP W$ given by $x\mapsto \log\abs{x}$ in the latter.
\end{rmk}


\begin{dfn}
By a \emph{square biflat} $F\subseteq G$ of a matroid $\M$ on $E$ we mean a pair $(F,G)$ of flats $F\in \L_{\M}$ and $G\in \L_{\M^\perp}$, such that and $(E\backslash F)|G$ is a bisubset: that is, $F\subseteq G$, and if $F=\emptyset$, then $\emptyset\neq G\subsetneq E$.
\end{dfn}


\begin{prp}\label{prop:conormal}
The square conormal fan $\Sigma_{-\M,\M^\perp}$ is the induced subfan of $\Sigma_{E,E}$ whose rays are indexed by square biflats $F\subseteq G$. 
In particular, $\Sigma_{-\M,\M^\perp}$ is smooth.
\end{prp}

\begin{proof}
The fan $-\Sigma_\M$ is an induced subfan of $\Sigma_E$ (see Example~\ref{exa:fans}.\eqref{exa:Bergman_fan}).
Combined with the identity $-e_F=e_{E\backslash F}$ for $F\in\L_{\M}$, the following Lemma~\ref{lem:induced} yields the claim.
\end{proof}


The next lemma makes use of the fact that any nonzero incidence vector $e_S$ in $\abs{\Sigma_E}$ spans a ray of $\Sigma_E$.

\begin{lem}\label{lem:induced}
For any two induced subfans $\Sigma_1$ and $\Sigma_2$ of $\Sigma_E$, $(\Sigma_1\times\Sigma_2)\cap\Sigma_{E,E}$ is the subfan of $\Sigma_{E,E}$ induced by the rays indexed by bisubsets $S|T$ for which $e_S\in\abs{\Sigma_1}$ and $f_S\in\abs{\Sigma_2}$.
\end{lem}

\begin{proof}
By hypothesis, $\Sigma_1\times\Sigma_2$ is an induced subfan of $\Sigma_E\times\Sigma_E$.
In particular, $\Sigma_1\times\Sigma_2$ consists of cones of $\Sigma_E\times\Sigma_E$, and the latter is refined
by $\Sigma_{E,E}$.
Each cone of $(\Sigma_1\times\Sigma_2)\cap\Sigma_{E,E}$ is thus a cone of $\Sigma_{E,E}$.

Suppose that $S|T$ is a bisubset such that $e_S+f_T$ spans a ray in $\abs{\Sigma_1\times\Sigma_2}=\abs{\Sigma_1}\times\abs{\Sigma_2}$. 
Then $e_S$ lies in a cone $\sigma_\bS$ of $\Sigma_1$ as in Example~\ref{exa:fans}.\eqref{exa:Bergman_fan}. 
Since $\Sigma_E$ is smooth, $e_S\in\ideal{e_{S_1},\dots,e_{S_k}}_\NN$ and hence $S\in\set{\emptyset,S_1,\dots,S_k}$.
Thus, $e_S$ is zero or spans a ray of $\Sigma_1$.

Let now $\sigma=\sigma_{\bS\vert\bT}$ be any cone of $\Sigma_{E,E}$ as in Example~\ref{exa:fans}.\eqref{exa:bipermutohedral_fan} with all its rays in $\abs{\Sigma_1\times\Sigma_2}$. 
Then $e_{S_1},\dots,e_{S_k}$ span a cone $\sigma_1$ of $\Sigma_E$.
By the above, it is a cone of the induced subfan $\Sigma_1$.
With a similar cone $\sigma_2$ of $\Sigma_2$, the cone $\sigma_1\times\sigma_2$ of $\Sigma_1\times\Sigma_2$ then contains $\sigma$.
This makes $\sigma$ a cone of $(\Sigma_1\times\Sigma_2)\cap\Sigma_{E,E}$ and the claim follows.
\end{proof}


\begin{cor}\label{cor:normalcones}
The cones $\sigma_{\bF\vert\bG}$ of $\Sigma_{-\M,\M^\perp}$ are indexed by biflags 
\[
\bF\vert\bG=\set{(E\backslash F_i,G_i)\mid 1\leq i\leq k}
\]
where $F_i\in\L_{\M}$ and $G_i\in\L_{\M^\perp}$.
Explicitly, this means that
\[
\begin{tikzcd}[row sep=small,column sep=small]
E \arrow[r,phantom,sloped,"\supsetneq"] & F_1
\arrow[r,phantom,sloped,"\supseteq"]\arrow[d,phantom,sloped,"\subseteq"]
& F_2
\arrow[r,phantom,sloped,"\supseteq"]\arrow[d,phantom,sloped,"\subseteq"]
& \cdots \arrow[r,phantom,sloped,"\supseteq"] & F_k
\arrow[d,phantom,sloped,"\subseteq"]\\ & G_1
\arrow[r,phantom,sloped,"\supseteq"] & G_2
\arrow[r,phantom,sloped,"\supseteq"] & \cdots
\arrow[r,phantom,sloped,"\supseteq"] & G_k
\arrow[r,phantom,sloped,"\supsetneq"] & \emptyset
\end{tikzcd}
\quad\text{and}\quad \bigcup_{i=1}^k(_i\backslash F_i\neq E.
\]
\end{cor}

\begin{proof}
Using the identity $-e_F=e_{E\backslash F}$, this follows from Proposition~\ref{prop:conormal} and a reformulation of the conditions on biflags in Example~\ref{exa:fans}\eqref{exa:bipermutohedral_fan}.
\end{proof}


\begin{ntn}\label{not:Delta_M_refined}
For any loopless matroid $\M$ on $E$, consider the smooth fan in $N_{E,E}$
\[
\wt\Delta_{\M}\coloneqq-\mu(\Sigma_{-\M,\M^\perp}).\qedhere
\]
\end{ntn}


Up to isomorphism, then, the square conormal fan refines $\Delta_{\M}$.


\begin{thm}\label{thm:normal_resolve}
The fan $\wt\Delta_{\M}$ defines a biprojective tropical resolution for $\M$.
\end{thm}

\begin{proof}
  We need to check that
  $\pi_1\times\pi_2\colon \wt\Delta_{\M}\to\Gamma_E\times(-\Gamma_E)$ is in
  fact a map of fans.  This can be seen from the following commutative diagram:
\[
\begin{tikzcd}
& \Gamma_E\\
\wt\Delta_{\M}\ar[ur,"-\pi_2"]\ar[two heads,d] & \Sigma_{-\M,\M^\perp}\ar[two heads,d]\ar[hookrightarrow,r]\ar{l}{\cong}[swap]{-\mu}\ar[u,"\delta",swap] & \Sigma_{E,E}\ar[two heads,d]\ar[ul,"\delta",swap] \\
\Delta_\M\ar[dr,"\pi_1",swap] & (-\Sigma_{\M})\times\Sigma_{\M^\perp}\ar[d,"-\pi_1"]\ar[hookrightarrow,r]\ar{l}{\cong}[swap]{-\mu} & \Sigma_E\times\Sigma_E\ar[dl,"-\pi_1"] \\
& \Gamma_E
\end{tikzcd}
\]
The arrows with hook denote induced subfans (see Example~\ref{exa:fans}.\eqref{exa:Bergman_fan} and Proposition~\ref{prop:conormal}), the ones with double head refinements (see Example~\ref{exa:fans}.\eqref{exa:bipermutohedral_fan} and Definition~\ref{dfn:normal}).
The left square is given by definition of $\Delta_\M$ and $\wt\Delta_{\M}$ (see Notations~\ref{not:Delta_M} and \ref{not:Delta_M_refined}).
In particular, $\abs{\wt\Delta_{\M}}=\abs{\Delta_\M}$.
The right square realizes Definition~\ref{dfn:normal} of $\Sigma_{-\M,\M^\perp}$.
The map $\delta$ is the defining feature of $\Sigma_{E,E}$ (see Example~\ref{exa:fans}.\eqref{exa:bipermutohedral_fan}).
Commutativity of the upper left triangle is due to the identity $\pi_2\circ\mu=\delta$ from Lemma~\ref{lem:m_mu}. 
The lower triangles exist trivially because the first component of $\mu$ is the identity.
\end{proof}

\subsection{Resolution of configuration hypersurfaces}\label{subsec:tropical_res}

In this last subsection, we examine the tropical resolution 
given by the square conormal fan.
By composing with the normalized Nash blow-up, we arrive at an explicit resolution of the configuration hypersurface, proving Theorem~\ref{thm:trop_res_intro} from the introduction.


\begin{dfn}\label{def:trop_lambda}
For any configuration $W\subseteq V$, set $\wt\Lambda_W\coloneqq\Lambda_W(\wt\Delta_{\M_W})$.
\end{dfn}


\begin{cor}\label{cor:trop_res}
Let $W\subseteq V$ be a connected configuration. 
Then $\wt\Lambda_W$ is a smooth schön compactification of $\Lambda_W^\circ\subseteq\TT_{E,E^\star}$ and the projection $\pi_2\colon N_{E,E}\to N_E$ induces a resolution of singularities of $X_W$, which factors through $\Lambda_W$:
\begin{equation}\label{eq:pqp2}
\begin{tikzcd}
\wt\Lambda_W\ar[rr,"q"]\ar[dr,"p",swap] && \Lambda_W\ar[dl,"p_2"] \\
& X_W
\end{tikzcd}
\end{equation}
\end{cor}

\begin{proof}
Combine Theorems~\ref{thm:trop_res} and \ref{thm:normal_resolve} using Corollary~\ref{cor:trop_res_factor}.
\end{proof}


As noted in the introduction, a key feature of a (sch\"on) tropical
compactification is its simple normal crossings boundary.
For each square
biflat $F\subseteq G$, there is a ray $e_{E\backslash F}+f_{G\backslash F}$ in
$\wt\Delta_W$ that indexes a torus-invariant divisor in $\PP(\wt\Delta_W)$
which we will denote $D_W^{F\subseteq G}$.
We further let $\wt D_W^{F\subseteq G}=D_W^{F\subseteq G}\cap \wt\Lambda_W$.


\begin{prp}\label{prop:trop_boundary}
Let $W\subseteq V$ be a connected configuration. 
Then
\[
\wt\Lambda_W\setminus\Lambda_W^\circ = \bigcup_{F\subseteq G} \wt D_W^{F\subseteq G},
\]
the union running over all  all square biflats $F\subseteq G$.
Each $\wt D_W^{F\subseteq G}$ is a smooth divisor, and 
$\bigcap_{1\leq i\leq k}D_{F_i\subseteq G_i}\ne\emptyset$ if and only if the bisubsets $(E\backslash F_i)\vert G_i$, $1\leq i\leq k$, form a biflag.
\end{prp}

\begin{proof}
By Corollary~\ref{cor:trop_res}, $\tilde\Lambda_W$ is schön.
By results of Tevelev and Hacking (see \cite[Thm.~2.4, Lem.~2.7]{Hac08}) it intersects the $\TT$-orbits in smooth varieties of expected dimension.

Under the schön hypothesis, the multiplication map in the toric variety is
smooth, which implies that the torus-invariant divisors in $\PP(\wt\Delta_W)$
intersect $\wt\Lambda_W$ transversely.  The divisor $D_W^{F\subseteq G}$ is
is a toric subvariety of $\PP(\wt\Delta_W)$, hence smooth; these facts together
imply that the intersection $\wt D_W^{F\subseteq G}$ is also smooth.
\end{proof}


The fibres of the maps $\wt\Lambda_W\to\Lambda_W\to X_W$ are controlled by the ambient toric variety.  
By definition, the first map is an isomorphism over the torus, and $\Lambda_W^\circ$ is smooth by Proposition~\ref{prop:Lambda_toric}. 
Now we consider its boundary by looking at its intersections with coordinate subspaces.

\begin{lem}\label{lem:strata}
Let $W\subseteq V$ be a configuration with connected matroid $\M=\M_W$. 
For any point $(w,\beta)\in\Lambda_W\setminus\Lambda_W^\circ$, the pair of sets $F(w)\subseteq F(w)\cup F(\beta)$ is a square biflat.
\end{lem}

\begin{proof}
We may work affinely, so let $(w,\beta)\in\hat\Lambda_W\backslash\hat\Lambda_W^\circ$ with $w\ne0$ and $\beta\ne0$.
Then $w\in W$ and $F(w)$ is a flat of $\M$.
By definition of $\hat\Lambda_W$, $Q_W^\star(w,\beta)=0$.
Equivalently, $w\beta:=Q^\star_E(w,\beta)\in W^\perp$, where
\[
(w\beta)_i=\beta_w(x_i)=\beta\circ Q_E(w,x_i)=\beta\left(\sum_{i\in E}w_ix_i\right)=\sum_{i\in E}w_i\beta_i.
\]
Thus, $F(w\beta)=F(w)\cup F(\beta)$ is a flat of $\M^\perp$.
If $F(w)=\emptyset$, then $F(w\beta)=F(\beta)$ is neither empty since $\beta\not\in\hat\TT_{E^\star}$, nor equal to $E$ since $\beta\neq 0$.
The conclusion follows.
\end{proof}


For any configuration $W\subseteq V$ and each pair $(F,S)\in \L_{\M_W}\times 2^E$ let 
\begin{align*}
\Lambda_W^{F,S} &\coloneqq \Lambda_W\cap \big(\PP(\KK^{E\backslash F})\times \PP(\KK^{E\backslash S})\big)\\
&= \set{(w,\beta)\in\Lambda_W\mid F(w)\supseteq F,\, F(\beta)\supseteq S}.
\end{align*}
Its preimage in $\wt\Lambda_W$ is an intersection with a toric variety:


\begin{prp}\label{prop:fibres} 
Let $W\subseteq V$ be a configuration with connected matroid $\M=\M_W$. 
Then the preimage under $q$ from \eqref{eq:pqp2} of $\Lambda_W^{F,S}$ for $(F,S)\in\L_{\M_W}\times 2^E$ is
\[
q^{-1}(\Lambda_W^{F,S})=\wt\Lambda_W\cap
\PP(\wt\Delta_\M^{F,S}),
\]
where $\wt\Delta_\M^{F,S}$ is the subfan of $\wt\Delta_\M$ induced on rays
indexed by square biflats $F'\subseteq G'$, where $F'\subseteq F$ and
$G'\backslash F'\subseteq S$.
\end{prp}

\begin{proof}
By construction, $\Lambda_W^{F,S}$ is the intersection of $\Lambda_W$ with the toric variety given by the star of the cone
\begin{equation}\label{eq:cone}
\sigma_{F,S}\coloneqq\RR_{\geq0}\big(\set{e_i\mid i\in F}\cup\set{f_j\mid j\in S}\big).
\end{equation}
Then $q^{-1}(\PP(\st(\sigma_{F,S})))$ is the toric variety given by the induced subfan of $\wt\Delta_\M$ on rays that intersect $\relint\sigma_{F,S}$ (see,  e.g., \cite[Lem.\ 3.3.21]{CLS11}). 
Rays of $\wt\Delta_\M$ are spanned by vectors  $e_{F'}+f_{G'\backslash F'}$, where $F'\subseteq G'$ is a square biflat. 
Such a ray intersects $\relint\sigma_{F,S}$ if and only if $F'\subseteq F$ and $G'\backslash F'\subseteq S$. 
The conclusion follows from Proposition~\ref{prop:trop_boundary}.
\end{proof}


\begin{exa}[Example~\ref{ex:delA3}, continued]
In this example, $\M=\M_W$ was not round, so $\Lambda_W$ was not smooth.
Here we describe the resolution $\wt\Lambda_W$ given by the square
conormal fan.  First, the (proper) flats of $\M$ and $\M^\perp$, respectively,
equal
\[
\set{1,2,3,4,5,124,135,23,25,34,45}\text{~and~}
\set{1,24,35},
\]
so the square biflats are
\[
\begin{array}{llllllll}
  i\subseteq E, & 124\subseteq E, & 135  \subseteq E, & 23 \subseteq E, & 25 \subseteq E,  & 34 \subseteq E, & 45 \subseteq E, &\\
  1\subseteq 1, & 2\subseteq 24, & 3\subseteq 35, & 4\subseteq 24, & 5\subseteq 25,
  & \emptyset\subseteq 1, & \emptyset\subseteq 24, & \emptyset\subseteq 35
\end{array}
\]
for $i=1,\ldots,5$.  The square biflat $F\subseteq G$ gives a ray
in the direction $-e_F+f_G$ in $\Sigma_{-\M,\M^\perp}$, which maps to
$e_F-f_{G\backslash F}$ under the the linear isomorphism
$-\mu\colon \Sigma_{-\M,\M^\perp}\cong \wt\Delta_{\M}$.
There are $56$ maximal cones.  

Fibres of $q\colon\wt\Lambda_W\to\Lambda_W$ can be understood by
restricting the map of ambient toric varieties. 
For example, we recall that $\Lambda_W$ has two singular points, $(\alpha_{124},\beta_0)$ and $(\alpha_{135},\beta_0)$, which are the strata $\Lambda_W^{124,2345}$ and $\Lambda_W^{135,2345}$ respectively, while
\[
\Lambda_W^{1,2345}=\set{(w_{st},\beta_0)\mid (s:t)\in\PP^1}
\]
is a line inside $\Lambda_W$, where
$w_{st}\coloneqq s\alpha_{124}+t\alpha_{135}$.

Using Proposition~\ref{prop:fibres}, we find four square biflats $\emptyset\subseteq 24$, $\emptyset\subseteq 35$, $1\subseteq 1$ and $1\subseteq E$
indexing the rays of the $2$-dimensional fan $\Delta_{W;1,2345}$, shown in bold in Figure~\ref{fig:delA3b}.  Then
$q^{-1}(\Lambda_W^{1,2345})$ 
is a union of four divisors in $\wt\Lambda_W$ that meet along
three curves.

At $t=0$, the fan $\Delta_{W;124,2345}\supseteq\Delta_{W;1,2345}$ is obtained
by adding three more rays indexed by $2\subseteq 24$, $4\subseteq 24$, and
$124\subseteq E$: the last ray is in every maximal cone.  This fan has dimension
$3$, so $q^{-1}(\set{(\alpha_{124},\beta_0)})$ is
a union of seven smooth divisors in $\Lambda_W$ that meet in pairs along 11
curves and in threes at five points.  For $s=0$, the situation is the same
up to symmetry.  The fans for the
toric varieties intersecting $q^{-1}(\Lambda_W^{1,2345})$
are shown schematically in Figure~\ref{fig:delA3b}.

\begin{figure}[ht]
\begin{tikzpicture}[scale=1.0,baseline=(current bounding box.center),
plain/.style={circle,draw,inner sep=1.5pt,fill=white},
root/.style={circle,draw,inner sep=1.5pt,fill=black},
every node/.style={circle,draw,minimum size=2mm,fill=black,inner sep=0pt},
    x={(1cm,-0.15cm)},
    y={(0.65cm,0.3cm)},
    z={(0cm,2.2cm)}  
]

  \coordinate (c)  at (0,0,0)   {};
  \coordinate (b)  at (0,-2,0)  {};
  \coordinate (lm) at (-2,1,0)  {};
  \coordinate (rm) at (2,1,0)   {};
  \coordinate (ll) at (-4.5,2,0)  {};
  \coordinate (lu) at (-2.5,-1,0)  {};
  \coordinate (rl) at (4,1,0)   {};
  \coordinate (ru) at (2,3,0)   {};
  \coordinate (nw) at (-1,0.5,1)  {};
  \coordinate (se) at (1,0.5,-1)  {};

  \node[label={[yshift=10pt]right:{\scriptsize $1\subseteq E$}}] at (c)   {};
  \node[label={below left:{\scriptsize $\emptyset\subseteq 35$}}] at (lm)  {};
  \node[label={above:{\scriptsize $\emptyset\subseteq 24$}}] at (rm)   {};
  \node[label={below:{\scriptsize $1\subseteq 1$}}] at (b)    {};
  \node[label={below left:{\scriptsize $3\subseteq 35$}}] at (ll)   {};
  \node[label={[xshift=5pt]below:{\scriptsize $5\subseteq 35$}}] at (lu)   {};
  \node[label={[xshift=5pt]right:{\scriptsize $2\subseteq 24$}}] at (rl)   {};
  \node[label={[xshift=5pt,yshift=-2pt]above:{\scriptsize $4\subseteq 24$}}] at (ru)   {};

  \fill[red,opacity=0.3] (nw) -- (c) -- (lm) -- cycle;
  \fill[red,opacity=0.3] (nw) -- (lm) -- (ll) -- cycle;
  \fill[red,opacity=0.3] (nw) -- (lm) -- (lu) -- cycle;
  \fill[red,opacity=0.3] (nw) -- (c) -- (rm) -- cycle;
  \fill[red,opacity=0.3] (nw) -- (c) -- (b) -- cycle;

  \fill[blue,opacity=0.3] (se) -- (c) -- (lm) -- cycle;
  \fill[blue,opacity=0.3] (se) -- (c) -- (rm) -- cycle;
  \fill[blue,opacity=0.3] (se) -- (rm) -- (rl) -- cycle;
  \fill[blue,opacity=0.3] (se) -- (rm) -- (ru) -- cycle;
  \fill[blue,opacity=0.3] (se) -- (c) -- (b) -- cycle;

  \draw[ultra thick] (c) -- (lm);
  \draw[ultra thick] (c) -- (rm);
  \draw[ultra thick] (c) -- (b);
  \draw (lm) -- (ll);
  \draw (lm) -- (lu);
  \draw (rm) -- (rl);
  \draw (rm) -- (ru);

  \draw (se) -- (c);
  \draw (se) -- (lm);
  \draw (se) -- (rm);
  \draw (se) -- (rl);
  \draw (se) -- (ru);
  \draw (se) -- (b);
  
  \draw (nw) -- (c);
  \draw (nw) -- (lm);
  \draw (nw) -- (rm);
  \draw (nw) -- (ll);
  \draw (nw) -- (lu);
  \draw (nw) -- (b);

  \node[fill=red, label={[xshift=1pt,yshift=5pt]right:{\scriptsize $135\subseteq E$}}] at (nw) {};
  \node[fill=blue, label=below right:{\scriptsize $124\subseteq E$}] at (se) {};

\end{tikzpicture}
\caption{Boundary structure of $q^{-1}(\Lambda_W^{1,2345})$}\label{fig:delA3b}
\end{figure}
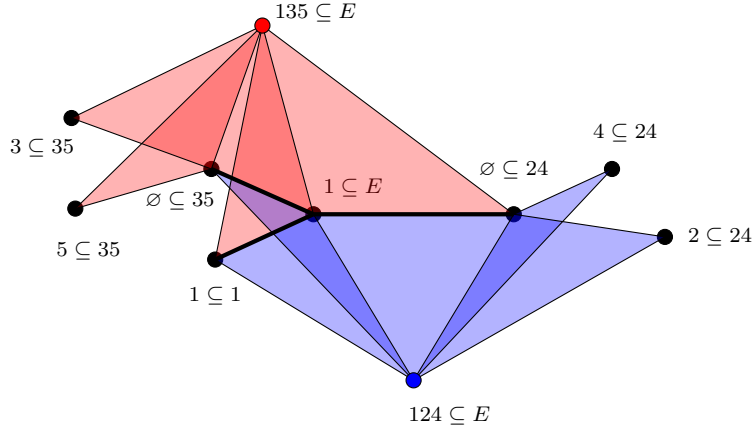
\end{exa}
\printbibliography
\end{document}